\documentclass[a4paper, 10pt]{article}
\usepackage{amssymb}
\usepackage{amsmath}
\usepackage{mathrsfs}
\usepackage{amsfonts}
\usepackage{color}
\usepackage{vmargin}
\usepackage{amsthm}
\usepackage{ esint }
\usepackage{a4}
\usepackage{graphicx}

\long\def\symbolfootnote[#1]#2{\begingroup%
\def\thefootnote{\fnsymbol{footnote}}\footnote[#1]{#2}\endgroup}

\setmarginsrb{20mm}{20mm}{20mm}{20mm}{10mm}{10mm}{10mm}{10mm}

\long\def\symbolfootnote[#1]#2{\begingroup%
\def\thefootnote{\fnsymbol{footnote}}\footnote[#1]{#2}\endgroup}

\def\vint{\mathop{\mathchoice%
          {\setbox0\hbox{$\displaystyle\intop$}\kern 0.22\wd0%
           \vcenter{\hrule width 0.6\wd0}\kern -0.82\wd0}%
          {\setbox0\hbox{$\textstyle\intop$}\kern 0.2\wd0%
           \vcenter{\hrule width 0.6\wd0}\kern -0.8\wd0}%
          {\setbox0\hbox{$\scriptstyle\intop$}\kern 0.2\wd0%
           \vcenter{\hrule width 0.6\wd0}\kern -0.8\wd0}%
          {\setbox0\hbox{$\scriptscriptstyle\intop$}\kern 0.2\wd0%
           \vcenter{\hrule width 0.6\wd0}\kern -0.8\wd0}}%
          \mathopen{}\int}

\newcommand{\loc}{\rm loc}
\newcommand{\Om}{\Omega}
\newcommand{\C}{\mathbb{C}}
\newcommand{\R}{\mathbb{R}}

\newcommand{\Hei}{{\mathbb{H}}_{1}}
\newcommand{\Hein}{{\mathbb H}_{n}}
\newcommand{\Hn}{{\mathbb H}_{n}}
\newcommand{\HW}{{HW}^{1,s}}

\newcommand{\HWtwloc}{{HW}^{1,2}_{loc}}
\newcommand{\HWloc}{{HW}^{1,s}_{loc}}

\newcommand{\HWzero}{{HW}^{1,s}_{0}}

\newcommand{\N}{{\mathcal N}}
\newcommand{\opL}{{\mathcal L}}
\newcommand{\g}{\mathfrak{g}}

\newcommand{\bd}{\partial}
\newcommand{\dist}{{\rm dist}}

\newcommand{\eps}{\epsilon}
\newcommand{\ep}{\epsilon}

\def \harm{\mathcal{H}}
\def \wharm{w\mathcal{H}}

\def \opl{\mathcal{L}}
\def \Hau{\mathcal{H}^3}
\def \dL{d_{\opL}}

\definecolor{blau}{rgb}{0.1,0.0,0.9}

\definecolor{funk}{rgb}{0.1,0.4,0.9}

\newcounter{komcounter}
\numberwithin{komcounter}{section}

\newcommand{\dbd}[2]{\frac{\partial#1}{\partial #2}}

\def\XXint#1#2#3{{\setbox0=\hbox{$#1{#2#3}{\int}$}
     \vcenter{\hbox{$#2#3$}}\kern-.5\wd0}}

\theoremstyle{plain}
\newtheorem{theorem}{Theorem}[section]
\newtheorem{thm}[theorem]{Theorem}
\newtheorem{lem}{Lemma}[section]
\newtheorem{cor}{Corollary}[section]
\newtheorem{prop}{Proposition}[section]
\newtheorem{observ}{Observation}[section]

\theoremstyle{definition}
\newtheorem{defn}{Definition}[section]
\newtheorem{ex}{Example}
\newtheorem{rem}{\textnormal{\textbf{Remark}}}

\begin{document}

\title {Mean value property and harmonicity on Carnot-Carath\'eodory groups}

\author{
Tomasz Adamowicz
\\
\it\small Institute of Mathematics, Polish Academy of Sciences \\
\it\small ul. \'Sniadeckich 8, 00-656 Warsaw, Poland\/{\rm ;}
\it\small T.Adamowicz@impan.pl
\\
\\
Ben Warhurst
\\
\it\small Institute of Mathematics,
\it\small University of Warsaw,\\
\it\small ul.Banacha 2, 02-097 Warsaw, Poland\/{\rm ;}
\it\small B.Warhurst@mimuw.edu.pl
}
\date{}
\maketitle

\footnotetext[1]{T. Adamowicz and B. Warhurst were supported by a grant Iuventus Plus of the Ministry of Science and Higher Education of the Republic of Poland, Nr 0009/IP3/2015/73.}

\begin{abstract}
 We study strongly harmonic functions in Carnot--Carath\'eodory groups defined via the mean value property with respect to the Lebesgue measure. For such functions we show their Sobolev regularity and smoothness. Moreover, we prove that strongly harmonic functions satisfy the sub-Laplace equation for the appropriate gauge norm and that the inclusion is sharp. We observe that spherical harmonic polynomials in $\Hei$ are both strongly harmonic and  satisfy the sub-Laplace equation. Our presentation is illustrated by examples.
\end{abstract}

\section{Introduction}

The main subject of our studies are harmonic functions on Carnot-Carath\'eodory groups with emphasis on the setting of Heisenberg groups since in this case the pseudodistance induced by the fundamental solution of the sub-Laplacian is in fact a metric (see below for relevant definitions). Following works \cite{agg,GG} we define harmonic functions via the mean value property with respect to the underlying measure, i.e. we call a locally integrable function $f:\Om\to \R$ \emph{strongly harmonic} in $\Om$,  if the following inequality holds for all balls $B(p, r)\Subset \Om$ with $p\in \Om$ and $r>0$:
 \begin{equation}\label{eq-intro}
  f(p)=\vint_{B(p,r)} f(q) dq.
  \end{equation}
 Here, $\Om$ stands for a domain in a given Carnot--Carath\'eodory group, $dq$ denotes the corresponding Lebesgue measure, and the balls $B(p, r)$ are defined with respect to a given metric on $\Om$. We refer to Section~\ref{subs-mvp} for further definitions and more on motivations for our investigations.

  Harmonic functions considered in \cite{agg} in general metric measure spaces are only H\"older regular, e.g., on geodesic spaces equipped with doubling measures or measures satisfying the annular decay condition, see \cite[Theorems 4.1, 4.2]{agg}, or locally Lipschitz regular for uniform measures or doubling measures on spaces supporting a $(1,p)$-Poncar\'e inequality, see \cite[Proposition 5.2, Theorem 5.1]{agg}. In the setting of Carnot groups, the group structure and the presence of the Euclidean coordinates allow us to expect that harmonic functions exhibit higher regularity properties. Indeed, in Section~\ref{subs-sob} we show that functions satisfying
  \eqref{eq-intro} belong to the horizontal Sobolev spaces $\HWloc$ for any $s>1$, see Theorem~\ref{thm-diff-harm-fun}. The proof relies on measure theoretic properties of harmonic functions. Furthermore, by using the convolution and scaling techniques available in Carnot groups, we show in Theorem~\ref{thm-mvp-smooth} the smoothness of harmonic functions. It turns out that for the proof of smoothness, one needs \eqref{eq-intro} to hold only for balls defined by a pseudodistance (quasimetric), i.e. the triangle inequality for $d$ in \eqref{eq-intro} can be relaxed.

  Another topic we are especially interested in, is the interplay between harmonic functions and solutions to the subelliptic Laplace equation on a Carnot group (called the $\opl$-harmonic equation). In Theorem~\ref{thm-mvp} we show that functions possessing property \eqref{eq-intro} satisfy the sub-Laplace equation provided that the balls in the mean value property are considered with respect to the pseudonorm given by the fundamental solution of the sub-Laplace operator. As a corollary of Theorem~\ref{thm-mvp} we obtain a variant of the Hadamard three-spheres theorem for strongly harmonic functions.  Let us also mention that a counterpart of Theorem~\ref{thm-mvp} in more general metric spaces is not known and is a subject of an ongoing investigation to determine the relation between strongly harmonic functions and the $p$-harmonic functions defined as local minima of the $p$-Dirichlet energy with respect to weak upper gradients.

  Another aspect of harmonicity studied in our work relates to the fact that the subelliptic harmonic functions are known to satisfy the kernel-type mean value property, see Formula~\eqref{mvpvol} in Theorem~\ref{thm-Lharm} below and Appendix for its proof. Thus, we are also interested in studying relation between this type of property and \eqref{eq-intro}, see Section~\ref{subs-conv}.

  In Section~\ref{sec-harm-hei1} we show that the intersection of the class of $\opl$-harmonic and strongly harmonic functions, considered with respect to the $\opl$-gauge distance, contains spherical harmonic polynomials, called for short, spherical harmonics. At the first glance it might be surprising that such a class exists, taking into account that a spherical harmonic must satisfy two types of mean value properties, namely the one given in Definition~\ref{defn-harm} and the one defined by \eqref{mvpvol}. Moreover, we discuss an example of a spherical harmonic function (and thus an $\opl$-harmonic function) which fails to be strongly harmonic, see Example~\ref{ex-sph}. Finally, in Section~\ref{sec-harm-hei1} we also propose two open questions on finding all spherical harmonics which are strongly harmonic.

  In the last section of our work we briefly discuss a notion of determining set and prove that under additional assumptions, a dense subset of a domain in $\Hei$ is determining for a strongly harmonic function, see Section~\ref{sec-det} for details.

  Our presentation is largely self contained and for the readers convenience in Sections~\ref{sec-prel}-\ref{sect-sublapl} we recall necessary definitions and observations regarding Carnot-Carath\'eodory groups, pseudonorms, subelliptic Laplacians and their fundamental solutions.

\section{Preliminaries}\label{sec-prel}

 In this section we recall some rudimentary properties of the geometry of Carnot-Carath\'eodory groups (CC-groups, for short). Upon recalling the definition of a CC-group, we illustrate the notions with examples of groups playing an important role in our studies, namely the $H$-type groups and the Heisenberg groups $\Hein$. Then in Section~\ref{sect-pseudo} we provide basic information about pseudonorms and pseudodistances. Definitions and results presented in that section will be used in our studies of the strongly harmonic functions, cf. Definition~\ref{defn-harm}, and their relations to the $\opl$-harmonic operator, see Section~\ref{sect5}. Finally, in Section~\ref{sect-conv} we recall the notion of convolution and provide its properties needed in our further presentation.

\subsection{Carnot--Carath\'eodory Groups}
 A Lie algebra $\mathfrak{g}$ is said to be stratified if $\mathfrak{g}= \mathfrak{g}_1 \oplus \dots \oplus \mathfrak{g}_s$, where  $\mathfrak{g}_{i+1} = [\mathfrak{g}_{1}, \mathfrak{g}_{i}]$ and $[\mathfrak{g}_1,\mathfrak{g}_s]=\{0\}$. The group $G=\exp(\mathfrak{g})$ is also said to be a stratified and we adopt the notation $\tau_p(q)=pq$ for the left translation of $q \in G$ by $p \in G$. If we choose an orthonormal basis for $\mathfrak{g}$, then the associated scalar product defines a Euclidean length on $\mathfrak{g}$, which we denote by $|X|$ for each $X \in \mathfrak{g}$. In this case, we call $G$ a Carnot--Carath\'eodory group (a Carnot group, for short) and $\mathfrak{g}$ a Carnot algebra. The normal model of $G$ is denoted $(\mathfrak{g}, *)$ where $*$ is given by the Baker--Campbell--Hausdorff formula, i.e., the exponential is an isomorphism. Dilation $\delta_\lambda$ of $\mathfrak{g}$ by $\lambda>0$ is given by   $\delta_\lambda(X)=\sum_i \lambda^i X_i$ where $X_i$ is the projection of $X$ onto $\mathfrak{g}_i$. An immediate consequence of the definition is that $\delta_\lambda \in {\rm aut} (\mathfrak{g})$ (=the automorphism group of the Lie algebra). Conjugating $\delta_\lambda$ with the exponential map defines dilation of $G$ which is again in ${\rm aut} (G)$, since for the normal model we have  $ {\rm aut} (G)$ and ${\rm aut}(\mathfrak{g})$ are one and the same.

We illustrate the above discussion with two main examples of the Carnot-Carath\'eodory groups. In what follows we will frequently appeal to these examples and assume that the reader is familiar with them.

\begin{ex}\label{Hn}
 The $n$-dimensional Heisenberg group  $G=\Hn$ is the Carnot group with a $2$-step Lie algebra and orthonormal basis $\{ X_1,\dots,X_{2n},Z\}$ such that
$$
\mathfrak{g}_1={\rm Span}\, \{X_1, \dots, X_{2n}\},  \quad \mathfrak{g}_2={\rm Span}\, \{Z\}
$$
and the nontrivial brackets are $[X_i,X_{n+i}]=Z$ for $i=1,\dots,n$.

In particular, if $n=1$, then a natural basis for the left invariant vector fields is given by the following vector fields:
\begin{align*}
\tilde X = \dbd{}{x}+2y\dbd{}{t}, \quad \tilde Y =\dbd{}{y}-2x\dbd{}{t} \quad {\rm and } \quad  \tilde T= \dbd{}{t},
\end{align*}
where $[\tilde X, \tilde Y ]=-4 \tilde T$. Note that these fields are defined with respect to the multiplication given by $V*W = V+W-4[V,W]$ for $V,W\in \mathfrak{g}$, which is not the Baker--Campbell--Hausdorff formula. The reason for choosing this slightly less orthodox multiplication is that it leads to a simpler expression for the Folland-Kaplan pseudonorm derived from the fundamental solution of $\mathcal{L} =\tilde X^2 + \tilde Y^2$ (see Example \ref{FolKapPseud} below). 
\end{ex}

\begin{ex}\label{H-type}
An $H$-type group is a connected, simply connected $2$-step Carnot
group whose Lie algebra satisfies the following additional property:
For each $Z \in \mathfrak{g}_2$ the homomorphism $J_Z: \mathfrak{g}_1 \to \mathfrak{g}_1$ defined by
$$
\langle J_Z X, Y \rangle = \langle Z, [X, Y] \rangle, \quad \hbox{for all }\, \, X,Y \in \mathfrak{g}_1
$$
satisfies
$$
J^2_Z = -\langle Z,Z\rangle I.
$$
\end{ex}

\begin{defn}[Change of Basis]\label{defn-basis} A basis $\mathcal{E}$ of $\mathfrak{g}$ is said to be \emph{adapted}, if it has the form $$\mathcal{E}=\{E_1^1,\dots,E_{N_1}^1, \dots, E_1^s,\dots,E_{N_s}^s\}, $$ where $\mathfrak{g}_k = {\rm span} \{ E_1^k, \dots,E_{N_k}^k \}$. Moreover, we define $N:=N_1+N_2+\ldots+N_s$.
\end{defn}

Note that if $ \tilde{\mathcal{E}}$ is also adapted to $\mathfrak{g}$, and $A$ is the transition matrix defined by $\pi_{ \mathcal{E} }(X)=\pi_{  \tilde{\mathcal{E}} }(AX)$ where $\pi_{ \mathcal{E} }$ and $\pi_{  \tilde{\mathcal{E}} }$ are the coordinate projections, then $A$ is a strata-preserving automorphism of $\mathfrak{g}$ and a strata-preserving isomorphism of $(\mathfrak{g}, *)$.

 The left translates of $ \mathfrak{g}_1$ define the horizontal subbundle $\mathcal{H} \subset TG$ which is naturally equipped with a left invariant sub-Riemannian metric $d_s$, defined by the left translation of the scalar product restricted to $\mathfrak{g}_1$. It is easy to see using normal coordinates, that a left Haar measure on $G$ is also a right Haar measure, and is simply the Lebesgue measure, up to scale, defined by the ambient Euclidean structure of the normal model.

If we choose a basis $\{E_1^1,\dots,E_{N_1}^1\}$ of $\mathfrak{g}_1$, then the left and right invariant vector fields corresponding to $E_i^1$ are defined as follows:
\begin{align*}
\tilde X_i^l u(p)= \frac{d}{dt} u(pe^{tE_i^1})|_{t=0} \qquad \tilde X_i^r u(p) = \frac{d}{dt} u(e^{tE_i^1}p)|_{t=0}. 
\end{align*}
The corresponding left and right invariant sub-Laplacians are
\begin{align}
\mathcal{L}^L u(p) =\sum_i {\tilde X_i^l} {\tilde X_i^l} u(p) \qquad \mathcal{L}^R u(p) =\sum_i {\tilde X_i^r} {\tilde X_i^r} u(p).
\label{sublap}
\end{align}
 \noindent \emph{We present our results in terms of left invariant fields and adopt the convention that unless otherwise stated, $\tilde X_i$ and $\mathcal{L}$ will be left invariant.}

\smallskip \noindent
We emphasize the role of the sub-Laplacians among the sub-elliptic operators in the following remark.
\begin{rem}
Let us consider a generalized sub-Laplacian, namely a left invariant operator of the form:
\begin{align}
L_Au=\sum_{ij}a_{ij}\tilde X_i \tilde X_j u, \label{L_A}
\end{align}
where $A=(a_{ij})$ is a symmetric and positive definite matrix with constant real coefficients. Set $B:=A^{1/2}$, then $B$ determines a strata preserving automorphism such that
\begin{align*}
L_A =\sum_k \tilde Y_k^2 
\end{align*}
where $\tilde Y_k = \sum_l b_{kl} \tilde X_l$. In other words, operator \eqref{L_A} can be reduced to the sub-Laplacian by a strata preserving automorphism.
\end{rem}

\subsection{Pseudonorm and Pseudodistance}\label{sect-pseudo}

The following definitions and remarks are crucial for our considerations and are largely applied in Sections~\ref{sect-sublapl}-\ref{sect5}. For further studies of pseudonorms and pseudodistances we refer, for instance,  to a book by Bonfiglioli--Lanconelli--Uguzzoni~\cite{blu}.

\begin{defn}\label{defn-pseudon}
A continuous function  $\mathcal{N} :G \to [0, \infty )$ is said to be a symmetric \emph{pseudonorm} on a Carnot group $G$, if $\mathcal{N}$ satisfies the following conditions:
\begin{enumerate}
\item $\mathcal{N}(\delta_r (p)) = r \mathcal{N}(p)$ for every $r > 0$ and $p \in G$,
\item $\mathcal{N}(p) > 0$ if and only if $p \ne 0$,
\item $\mathcal{N}$ is symmetric, i.e., $\mathcal{N}(p^{-1})=\mathcal{N}(p)$ for every $p \in G$.
\end{enumerate}
\end{defn}

All pseudonorms of interest here will be symmetric, and so their symmetry will not be emphasized in their reference.
\begin{ex}\label{psdnrm1}
	Let $G$ be a group expressed in coordinates by choosing an orthonormal basis of $\g$. Let us further write any point $p\in G$ as $p=\sum_{i=1}^s P_i$, where $P_i \in \mathfrak{g}_i$. Then it follows that for every $\alpha \geq 1$, a function
\begin{align}
\mathcal{N}(p)=\left( \sum_{i=1}^s |P_i(p)|^{\alpha/i}  \right )^{1/\alpha} \label{canonN}
\end{align}
 defines a pseudonorm on $G$.  For an appropriate choice of $\alpha$, $\mathcal{N}$ is $C^\infty $ on $G\setminus\{0\}$, e.g., $\alpha=2s!$.
\end{ex}

\begin{ex} \label{FolKapPseud}
Let $G=\Hei$ be the first Heisenberg group. Then, the Folland-Kaplan pseudonorm derived from the fundamental solution of $\mathcal{L} = \tilde X^2 + \tilde Y^2$ is given by
\begin{align*}
\mathcal{N}(z,t) =  ( |z|^4 + t^2)^{1/4}, 
\end{align*}
where coordinates of any point $p\in \Hei$ are $p=(z,t)$ with $z\in \C$ and $t\in \R$. It follows that
\begin{align*}
|\nabla_0 \mathcal{N}(z,t)|^2=\frac{|z|^2}{\sqrt{|z|^4 + t^2}}.
\end{align*}
\end{ex}

It turns out that a pseudonorm defines a pseudodistance (quasimetric) on a group.

\begin{defn}\label{defn-pseudod}
We say that a left invariant pseudodistance (quasimetric) is \emph{induced} by a pseudonorm $\mathcal{N}$, if
$$
d(p,q):=\mathcal{N}(p^{-1}q),\qquad\hbox{for any }p,q\in G.
$$
In particular:
\begin{itemize}
\item[(i)] $d(p,q)\geq 0$ with equality if and only if $p=q$,
\item[(ii)] there exists a constant $c>0$, such that
   $$
    d(p_1,p_2)\leq c (d (p_1,p_0) +d(p_0,p_2)) \quad \hbox{ for all } p_0,p_1,p_2 \in G.
   $$
\end{itemize}
\end{defn}

Although pseudodistances are not metrics, they still define a reasonable family of sets that play the role of balls.

\begin{defn} Let $d$ be a given pseudometric on $G$. The unit ball with respect to $d$, centered at $0\in G$ is denoted by $B(0,1)$, and a ball of radius $r$ centered at $p\in G$ is the set
$$
B(p,r)=p\delta_r(B(0,1)).
$$
Note that it follows that $|B(p,r)|=r^Q |B(0,1)|$, where $Q=\sum_i i\, {\rm dim}\, \mathfrak{g}_i$ is the Hausdorff dimension of $G$.
\end{defn}

We finish the presentation of basic properties of pseudonorms and pseudodistances with the observation that all pseudonorms on a given group $G$ are equivalent.
\begin{rem}[cf. Proposition 5.1.4 in \cite{blu}]\label{rem-psn}
For any pair of pseudo-norms $\mathcal{N}$ and $\mathcal{N}'$ on $G$, there exists $c>0$ such that
$$
\frac{1}{c} \mathcal{N}(p) \leq \mathcal{N}'(p) \leq c\,\mathcal{N}(p).
$$
 Moreover, if $d_s$ stands for a sub-Riemannian distance in $G$, then $\mathcal{N}'(p)=d_s(0,p)$ is a pseudonorm and so $d$ and $d_s$ are equivalent pseudodistances.
\end{rem}

\subsection{Convolution}\label{sect-conv}

 Below we recall the notion of the convolution and some of its integrability and regularity properties. The section is based on presentation in Chapter 1 in Folland--Stein~\cite{FolStn}.

Let $G$ be a stratified group. The convolution of two measurable functions $g$ and $h$ on $G$ is defined as follows:
$$
g * h(p)=\int_{G} g(q)h(q^{-1}p)dq=\int_{G} g(pq^{-1})h(q)dq,
$$
provided the integrals converge. In contrast with Euclidean spaces, the convolution is not commutative. The basic properties are as follows, cf. Propositions 1.19, 1.20 and pg. 22 in \cite{FolStn}:
\begin{enumerate}
\item Let $1\leq r,s,t < \infty$ such that $ r^{-1}+s^{-1}=t^{-1}+1$. If $g \in L^r(G)$ and $h \in L^s(G)$, then $g * h \in L^t(G)$ and $\|g * h\|_t \leq c(r,s)\|g \|_r\, \|h\|_s$.
\item If $L$ is a left invariant vector field, then $L(g * h)=g * (Lh)$ provided that $h$ is $C^2$-smooth.
\item If $R$ is a right invariant vector field, then $R(g * h)=(Rg) * h$ provided that $g$ is $C^2$-smooth.
\item If $L$ and $R$ are corresponding left and right invariant vector fields (i.e., they agree at the identity), then  $(Lg) * h = g * (Rh)$ when $g$ and $h$ are suitably smooth.
\item If $\psi \in L^1(G)$ and $\psi_\eps =\eps^{-Q} \psi \circ \delta_{1/\eps}$, then $\int_{G} \psi_\eps(q)dq=\int_{G} \psi(q) dq $.
\item If $h \in L^r(G)$, $1 \leq r <\infty$ and $\int_{G} \psi(q)dq=c$, then $\|h * \psi_\eps -ch\|_r \to 0$ as $\eps \to 0$.
\item If $h$ is continuous and bounded on $\Omega$, then $h * \psi_\eps\to ch$ locally uniformly on $\Omega$ as $\eps \to 0$.
\end{enumerate}

For the sake of completeness of the discussion, we further recall that convolutions can be defined also in the setting of distributions. Namely, we set
$(\Lambda *g,\phi) =(\Lambda, \phi * \tilde g)$ and $(g*\Lambda,\phi) =(\Lambda,  \tilde g * \phi )$ for any $\phi\in C^{\infty}_{0}(G)$, where $\tilde g(p):=g(p^{-1})$. Therefore, it then follows that if $\Lambda_u$ is a distribution defined as $(\Lambda_u,\phi)=\int_G u \phi $, then
\begin{align*}
(\Lambda_u *g,\phi) &=(\Lambda_{u*g} , \phi ) \quad {\rm and} \quad ( g*\Lambda_u,\phi) =(\Lambda_{g*u}, \phi ). 
\end{align*}

\section{Sub-Laplacians, Fundamental solutions and Mean value property}\label{sect-sublapl}

 The first goal of this section is to recall the weak formulation of the sub-Laplacian together with a notion of the fundamental solution and its relation with pseudonorms. Then, in Section~\ref{subs-mvp} we give the definition of strongly harmonic functions, one of the main objects studied in this work. Such functions are defined originally in \cite{agg} in the setting of metric measure spaces.

 We begin with recalling the horizontal Sobolev spaces and the weak formulation of the $\opl$-harmonic equation, cf. \eqref{sublap} and \eqref{L_A}.

  Let $\Omega \subset G$ be open. Recall that a $C^2(\Omega)$-smooth function $u$ is said to be $\mathcal{L}$-harmonic on $\Omega$, if $\mathcal{L}u(p)=0$ for all $p\in \Omega$, cf.~\eqref{sublap}. In order to weaken the $C^2$-regularity assumption one considers $\mathcal{L}$ on an appropriate Sobolev space.

Recall that $N_1$ denotes the dimension of $\mathfrak{g}_1$, cf. Definition~\ref{defn-basis}. For $1<s<\infty$, we say that a function $u:\Omega\to \R$ belongs to the \emph{horizontal Sobolev space} $\HW(\Omega)$, if $u\in L^s(\Om)$ and for all $i=1,\ldots, N_1$, the horizontal derivatives $\tilde X_i u$ exist in the distributional sense and are represented by elements of $L^s(\Om)$. The space $\HW(\Omega)$ is a Banach space with respect to the norm
\[
\|u\|_{\HW(\Omega)}\,=\,\|u\|_{L^s(\Omega)}+\|(\tilde X_1 u,\ldots, \tilde X_{N_1} u)\|_{L^s(\Omega)}.
\]
In the similar way we define the local spaces $\HWloc (\Omega)$. Moreover, we define $\HWzero(\Omega)$ as the closure of $C_{0}^{\infty}(\Omega)$ in the $\HW(\Omega)$-norm. The horizontal gradient $\nabla_0 u$ of an element $u \in \HWloc (\Omega)$ is defined by 
\[
\nabla_0 u = \sum_{i=1}^{N_1} ( \tilde X_i u) \tilde X_i.
\]

If $u \in \HWtwloc (\Om)$, then $u$ is said to be  weakly  $\mathcal{L}$-harmonic on $\Omega$ if
\begin{equation*}
  (\mathcal{L} \Lambda_u, \phi):=\int_{\Om} u(q) \mathcal{L} \phi(q) dq = -\int_{\Om} \langle \nabla_0 u(q), \nabla_0 \phi(q) \rangle dq =0
\end{equation*}
 for all $\phi \in C_0^\infty(\Om)$.

\subsection{Fundamental solutions}

  The notion of the fundamental solution of the Laplace operator $L=\Delta$ is well known in the setting of domains in $\R^n$. Among many applications, fundamental solutions are used to solve the related Poisson equation $\Delta u=f$ via the Newtonian potentials of $f$ when $f$ is sufficiently regular, see e.g. Chapter 4 in Gilbarg--Trudinger~\cite{gt} for the discussion of this classical topic.

Let us also recall the related notion of a \emph{hypoelliptic} operator. We say that a differential operator $P$ with smooth coefficients, defined on an open subset $\Om\subset \R^n$ is called hypoelliptic, if for every distribution $u$ defined on an open subset $\Om'\subset \Om$ satisfying $Pu\in C^{\infty}(\Om')$, it holds that also $u\in C^{\infty }(\Om')$. The connection between the hypoellipticity and the fundamental solutions can be formulated as follows (see e.g. H\"ormander~\cite{hor}): Let $\{ \tilde Y, \tilde X_1 ,\dots, \tilde X_m\}$ be smooth vectors fields on $\Omega\subset \R^n$. Suppose
$$
\hbox{rank Lie}\, \{\tilde Y,\tilde X_1 ,\dots, \tilde X_m\}(x) = n, \quad \hbox{for all }x \in \Omega.
$$
Then the operator
$$
L =\tilde Y + \sum_{j=1}^m \tilde X_j^2
$$
is $C^\infty(\Omega)$-hypoelliptic. Thus, every distributional solution to $Lu = f$ is represented by a function in $C^\infty(\Omega)$ when $f \in C^\infty(\Omega)$.
%

We move now our discussion of fundamental solutions to the setting of CC-groups.

\begin{defn}[Fundamental solution] Let $G$ be a stratified Lie group and let $\mathcal{L}$ be a sub-Laplacian on $G$. A function $\Gamma : G \setminus {0} \to \R$ is called a \emph{fundamental solution} for $\mathcal{L}$ if:
\begin{itemize}
\item[(i)] $\Gamma \in C^\infty(G \setminus \{0\})$,
\item[(ii)] $\Gamma \in L^1_{\loc} (G)$ and $\Gamma (q) \to 0$, when $q \to \infty$,
\item[(iii)] $\mathcal{L} \Gamma = {\rm Dirac}_0$, where ${\rm Dirac}_0$ is the Dirac measure supported on $\{0\}$. More explicitly
\begin{equation}
\int_G \Gamma(q) \mathcal{L}\phi(q)dq = \phi(0), \quad \hbox{ for all }\,\phi \in C_0^\infty(G). \label{phi0}
\end{equation}
\end{itemize}
\end{defn}

The following result states that a sub-Laplace operator possesses a fundamental solution.

\begin{thm}[Theorem 2.1 in Folland~\cite{Fol}, Theorem 5.3.2 in \cite{blu}]\label{thm-fund-exist} Let $\mathcal{L}$ be a sub-Laplacian on $G$. Then there exists a fundamental solution $\Gamma$ for $\mathcal{L}$ satisfying \eqref{phi0}.
\end{thm}

The relation between the $\opl$-harmonicity and the pseudodistances plays important role in our discussion, see Section~\ref{sect5}. For this reason, we now recall, following Definition 5.4.1 in \cite{blu}, that a homogeneous symmetric form $d$ on $G$ is an \emph{$\opl$-gauge}, if
\begin{equation}
   \opl(d^{2-Q})=0\qquad \hbox{in }G\setminus\{0\}. \label{def-lgauge}
\end{equation}
Using this notion we may observe that pseudonorm, pseudodistances and the fundamental solution are related to each other in the following way.
\begin{thm}[Proposition 5.4.2 and Theorem 5.5.6 in \cite{blu}]\label{thm-fund-psn}
 Let $\Gamma$ be the fundamental solution of a sub-Laplacian $\opl$ defined on a group $G$ of homogeneous dimension $Q>2$. Then,
 \begin{equation}
  \mathcal{N}(p)=
  \begin{cases}
  -\Gamma^{1/(Q-2)}(p)\qquad  &p\in G\setminus\{0\}, \label{eq-N-fund}\\
  0\qquad &p=0.
  \end{cases}
 \end{equation}
  is a pseudonorm and hence, defines a pseudodistance (see Defintion~\ref{defn-pseudod}).

  Furthermore, the opposite relation holds as well: an $\opl$-gauge $d$ defines a fundamental solution $\Gamma=C_d d^{2-Q}$. Therefore, $\opl$-gauge is unique up to a constant.
\end{thm}

We illustrate the above theorems with an example of the $H$-type groups.
\begin{ex}\label{ex-heis}
 Basing on Example~\ref{H-type}, we recall that a Lie algebra $\mathfrak{g}$ of an $H$-type group decomposes as
 follows: $\mathfrak{g}=\mathfrak{g}_1 \oplus \mathfrak{g}_2$. If $\pi_i :\mathfrak{g} \to \mathfrak{g}_i$, $i=1,2$, are the natural projections, then the \emph{Folland-Kaplan pseudonorm} on the normal model $(\mathfrak{g}, *)$ is given by the following formula:
\begin{equation*}\label{a-dist}
\mathcal{N}(X):=a(X)^{\frac14}:=\left(\langle \pi_1(X),\pi_1(X) \rangle^2 + 16
\langle \pi_2(X), \pi_2(X) \rangle \right)^\frac14,
\end{equation*}
and so the fundamental solution for the $\opl$-Laplace operator takes a form:
$$
-\Gamma(X)=c\mathcal{N}(X)^{2-Q},
$$
where $c=c(Q)>0$.
\end{ex}

Finally, we recall a representation formula for test functions in terms of the sub-Laplacian, see Theorem 5.3.3 in \cite{blu}. Namely, by applying Formula~\eqref{phi0} to a test function $G\ni q \to \phi(pq)$ defined for any given $\phi\in C_0^{\infty}(G)$ and any $p\in G$, gives the following:
$$
 \int_G \Gamma(p^{-1}q) \mathcal{L}\phi(q)dq =-\phi(p),
$$
where $\Gamma$ is a fundamental solution for $\mathcal{L}$ in $G$. A similar formula arises in the studies of the
non-homogeneous sub-Laplace equation $\mathcal{L}u=f$ with $f \in C_0^\infty(G)$. Then, it holds:
$$
 u(p)=f*\Gamma(p)=\int_G f(q) \Gamma(q^{-1}p)dq.
$$
 For further studies about this equation and the associated linear potentials we refer, for instance, to the book of Ricciotti~\cite{Ricciotti}.

\subsection{Mean value property - strongly and weakly harmonic functions}\label{subs-mvp}

In the previous sections we recall and discuss the harmonicity from the point of view of the subelliptic Laplacian $\opl$. Below we introduce a different approach based on the mean value property, the so-called \emph{strongly} and \emph{weakly} harmonic functions. Such functions are defined in the general setting of metric measure spaces, and were introduced in Gaczkowski--G\'orka~\cite{GG} and studied further in \cite{agg}, where several rudimentary properties of such functions are proved. These results include the maximum and comparison principles, various types of Harnack estimates, local H\"older and Lipschitz regularity, the Liouville type theorems and the solvability of the Dirichlet problem based on the dynamical programming method and the Perron method. The interplay between the underlying measure and metric turns out to be crucial in such studies. Furthermore, the flexibility in choosing the distance function and measures sheds new light on harmonic functions.

Following the discussion in Section 3 in \cite{agg}, we recall the key definitions specialized to the setting of Carnot-Carath\'eodory groups. In this section, by $d$, we denote a metric induced by a left-invariant norm $\mathcal{N}$ on a Carnot-Carath\'eodory group $G$, cf. Section~\ref{sect-pseudo}. A ball $B(p, r)$ in $G$ centered at $p\in G$ with radius $r>0$ is defined as follows:
\[
 B(p, r)=\{q\in G: d(p, q)\leq r\}.
\]

\begin{defn}\label{defn-harm}
 Let $\Om\subset G$ be an open set. A locally integrable function $f:\Om\to \R$ is called \emph{(strongly) harmonic} in $\Om$ if the following inequality holds for all balls $B(p, r)\Subset \Om$ with $p\in \Om$ and $r>0$:
 \[
  f(p)=\frac{1}{|B(p,r)|}\int_{B(p,r)} f(q) dq.
  \]
 Here, $dq$ stands for the Lebesgue measure. The set of all harmonic functions in $\Om$ will be denoted $\harm(\Om)$.
\end{defn}

The following definition generalizes the previous one and strengthens motivation for studying harmonic functions defined via the mean value property, see the discussion below.
 \begin{defn}\label{defn-w-harm}
 Let $\Om\subset G$ be an open set. A locally integrable function $f:\Om\to \R$ is called \emph{weakly harmonic} in $\Om$ if for every $p\in \Om$ there exists a non-empty set of positive radii $\{r^p_\alpha\}_{\alpha\in I}$, for some indexing set $I$, such that the following inequality holds for all balls $B(p, r^p_\alpha)\Subset \Om$:
 \[
  f(p)=\frac{1}{|B(p, r^p_\alpha)|}\int_{B(p,r^p_\alpha)} f(q) dq.
  \]
 As in the previous definition, $dq$ denotes the Lebesgue measure. The set of all weakly harmonic functions in $\Om$ will be denoted $\wharm(\Om)$.
\end{defn}

 Let us also remark that the above definitions are robust enough to allow us to study the mean value property not only for various measures and distances but in fact for pseudodistances as presented in Section~\ref{subs-smooth} below. Indeed, it turns out that the $C^{\infty}$-regularity of functions in $\harm$ holds also if $d$ is a pseudodistance.

 Motivation for studying relations between the harmonicity and the mean value property in the Euclidean setting has for two centuries been an ongoing subject of research, and it was Gauss who, perhaps first, showed that harmonic functions have the mean value property. The opposite question, whether all radii of balls contained in an underlying domain are needed for the mean value property to hold has also been investigated by several prominent mathematicians, such as Koebe, Volterra and Kellogg, Hansen and Nadirashvili, Blaschke, Privaloff and Zaremba. Let us also remark, that several results in this direction of studies are known as \emph{one radius (two radii) theorems}. We refer to \cite{agg} for further historical comments and references regarding the mean value property and harmonicity. The mean value property continues to inspire studies in PDE, leading for instances to $p$-harmonious functions and their stochastic tug-of-war games, see \cite{agg}.

\section{Regularity of weakly and strongly harmonic functions and their relations with $\opl$-harmonic functions}\label{sect5}

The main results of the paper are presented in this section. Among other topics, we study regularity properties of weakly and strongly harmonic functions. First we show that the familiar method of difference quotients applies to weakly and strongly harmonic functions and obtain the local Sobolev regularity of such functions, see Theorem~\ref{thm-diff-harm-fun}. We then prove that functions satisfying Definition~\ref{defn-harm} in Carnot groups are smooth (see also Remark~\ref{rem-weak-smooth} commenting the lack of a similar property for weakly harmonic functions). This observation is then used in our studies of relations between the strongly harmonic functions and the canonical sub-Laplacian  $\opl$. Namely, by using Theorem~\ref{thm-mvp-smooth}, we show that strongly harmonic functions defined with respect to the pseudodistance induced by the fundamental solution for $ \opl$ (which in the particular example of the Heisenberg group is a metric) satisfy the $\opl$-harmonic equation, i.e., they are $\opl$-harmonic, see Theorem~\ref{thm-mvp}.

We present both: Sobolev and $C^{\infty}$-regularity properties of strongly harmonic functions. The reason for this is that techniques employed in these studies differ from each other. The Sobolev type regularity is based on the measure theoretic properties of functions in the class $\harm$, while the smoothness result relies on their analytic properties and the structure of Carnot-Carath\'eodory groups, in particular with respect to convolutions. Furthermore, for the smoothness of strongly harmonic functions we may weaken the assumptions on the underlying distance and allow it to be merely a pseudodistance. This observation generalizes definitions presented in \cite{agg} and is possible here as the more general definition of a harmonic function is now compensated by the presence of group and Euclidean structures as in $\Hein$. Therefore, results of Sections~\ref{subs-sob} and \ref{subs-smooth} illustrate the richness of the structure of functions in $\harm$ and provide additional motivation for further investigation. 

\subsection{Local Sobolev regularity}~\label{subs-sob}
Let $Z$ be a vector field in a Carnot group $G$ and $u:\Om\to \R$ be a function from a domain $\Om\subset \Hn$. The \emph{difference quotient} of $u$ at $p\in \Om$ is defined as follows:
 \[
  D_h^Z(p)=\frac{u(pe^{hZ})-u(p)}{h}\,=\, \frac{u(pe^{\delta_h(Z)})-u(p)}{h}, \quad \hbox{ when} \quad Z \in \mathfrak{g}_1 \quad \hbox{ and } h\not=0.
 \]
 We mention that the second quotient above is the Pansu difference quotient and that in general the second equality is false when $Z \not \in \mathfrak{g}_1$. Recall that for any $V \in \mathfrak{g}$, the Pansu derivative is given by
$$ PDf(p)(V)=\lim_{h \to 0} \frac{u(pe^{\delta_h(V)})-u(p)}{h}$$ and is in some sense the natural derivative to consider on Carnot groups. In Lemma 3.6 of Capogna-Cowling \cite{capcow}, it is shown that $Q$-($\mathcal{L}$-harmonic) functions are Pansu differentiable almost everywhere. Note here that $Q$ is the Hausdorff dimension of $G$  and $u$ is $Q$-($\mathcal{L}$-harmonic) if and only if
$$\int_\Omega |\nabla_0 u(x)|^{Q-2}\nabla_0 u(x) \cdot \nabla_0 \phi (x) dx=0$$
for every $\phi \in C_0^\infty(\Omega)$. Note that $\mathcal{L}$-harmonic functions are $ 2$-($\mathcal{L}$-harmonic).

The following result holds in the Carnot group setting similarly to the Euclidean setting, cf. Gilbarg--Trudinger~\cite[Chapter 7.11]{gt}, H\"ormander~\cite{hor}, Manfredi--Mingione~\cite{minm}, Capogna~\cite{cap} and Ricciotti~\cite{Ricciotti}.

\begin{lem}[cf. Lemma 1 in \cite{minm}]\label{lem-diff-quot}
  Let $\Om\subset \Hn$ be a domain and $\Om'\Subset \Om$. Furthermore, let $\tilde Z$ be a left-invariant vector field corresponding to $Z \in \mathfrak{g}_1$ and let $u\in L^s_{loc}(\Om)$ for $s>1$. If there exist constants $\sigma<\dist(\Om', \partial \Om)$ and $C>0$  such that
  \[
   \sup_{0<|h|<\sigma} \int_{\Om'}|D_h^Zu(q)|^s dq\leq C^s,
  \]
  then $\tilde Zu\in L^s(\Om')$ and $\| \tilde Zu\|_{L^s(\Om')}\leq C$. Conversely, if $\tilde Zu\in L^s(\Om')$, then for some $\sigma>0$ it holds that
  \[
   \sup_{0<|h|<\sigma} \int_{\Om'}|D_h^Zu(q)|^s dq\leq \left(2\|\tilde Zu\|_{L^s(\Om')} \right)^s.
  \]
  Moreover, if the first assertion holds, then $D_h^{Z}u$ converges strongly to $\tilde Zu$ in $L^s(\Om')$, as $h\to 0$.
\end{lem}

For another proof of the lemma, we refer to \cite[Theorem 2.11]{Ricciotti}, where it is stated and proved for the Heisenberg group. However, one easily observes that the proof holds for any Carnot group as well. We use the lemma to show the following result on the Sobolev regularity of strongly harmonic functions.

 \begin{thm}\label{thm-diff-harm-fun}
 Let $\Om\subset G$ be a domain and let $\Om'\Subset \Om$ with $\dist(\Om', \bd \Om)>0$. Suppose that $u$ is a harmonic function in $\harm(\Om, \mu)$, where $\mu$ stands for the Lebesgue measure on $G$ and the underlying metric $d$ is such that $G$ is a geodesic space in $d$. Let $X \in \mathfrak{g}_1$, then $\tilde Xu \in L^s(\Om')$ for all $s>1$ and $u\in HW^{1,s}(\Om')$.

 In particular, if $G=\Hein$, and $\tilde X_i$ for $i=1,2,\ldots, 2n$ are the left-invariant horizontal vector fields, then $\tilde X_iu \in L^s(\Om')$ for all $s>1$ (see also the discussion following Definition~\ref{defn-basis} and Example~\ref{ex-heis}).

The assertion holds as well if $u$ is weakly harmonic function in $\wharm(\Om, \mu)$, provided that for any $p\in \Om'$ the mean value property holds for $u$ with respect to ball $B(p, \frac12\dist(\Om',\bd \Om))$.
 \end{thm}

 In the proof below we appeal to a definition and some results from the general theory of strongly harmonic functions on metric measure spaces, developed in \cite{agg}. For the readers convenience we state them now.

 Let $(X, d, \mu)$ be a metric measure space. We say that a measure $\mu$ \emph{is continuous with respect to a metric $d$} if for all $x\in X$ and all $r>0$ it holds that
 \begin{equation}\label{m-cont-m}
  \lim_{\Om \ni y\underset{d}\to x} \mu(B(x,r) \Delta B(y,r))=0,
 \end{equation}
 where $\Delta$ denotes the symmetric difference of two sets.

\begin{prop}\label{prop-dbl}
 (a) Let $(X, d, \mu)$ be a geodesic doubling metric measure space. Then $\mu$ is continuous with respect to the metric $d$ (see Proposition 2.1, \cite{agg}).\\
 (b) Let $(X, d, \mu)$ be a metric measure space with measure $\mu$ continuous with respect to the metric $d$. If $f\in \harm(\Om, \mu)$, then $f$ is locally bounded in $\Om$. (Corollary 4.1 in \cite{agg}).
\end{prop}
 \begin{proof}[Proof of Theorem~\ref{thm-diff-harm-fun}]

  The proof appeals to the method presented, for instance, in the proof of Proposition 4.1 in \cite{agg}. Let us first observe that since $(G, d, \mu)$ is, by its properties, a geodesic doubling space, then by Proposition~\ref{prop-dbl}(a), $\mu$ is continuous with respect to the sub-Riemmanian metric and therefore Proposition~\ref{prop-dbl}(b) implies the local boundedness of $u$ in $\Om$.

  Let $h\in \R$ be such that $|h|<\dist(\Om', \bd \Om)$. Note that since $X \in \mathfrak{g}_1$ we have
  \[
   d(p, pe^{hX})= d(p,e^{\delta_{h}(X)})=|h|\mathcal{N}(X),
   \]
 where the norm $\mathcal{N}$ is such that $d(p,q)=\mathcal{N}(p^{-1}q)$ for all $p,q\in G$.  The invariance of the measure implies that for any $X \in \mathfrak{g}_1$, and any $r>0$ such that $B(p, r)\Subset \Om'$ and any $|h|<r/\mathcal{N}(X)$, we have
  \begin{align}
  |u(pe^{hX})-u(p)|&=\left|\frac{1}{|B(pe^{hX}, r)|}\int_{B(pe^{hX},r)} u(q) dq-\frac{1}{|B(p, r)|}\int_{B(p,r)} u(q) dq\right| \nonumber \\
  &=\Bigg|\frac{1}{|B(pe^{hX}, r)|}\int_{B(pe^{hX}, r)} u(q) dq-\frac{1}{|B(pe^{hX}, r)|}\int_{B(p, r)} u(q) dq \nonumber  \\
  &-\frac{|B(pe^{hX}, r)|-|B(p, r)|}{|B(pe^{hX}, r)|\,|B(p, r)|}\int_{B(p, r)} u(q) dq\Bigg|\nonumber  \\
  &\leq \frac{1}{|B(pe^{hX}, r)|} \int_{B(pe^{hX}, r) \vartriangle B(p, r)} |u(q)| dq
  + \frac{\left||B(pe^{hX}, r)|-|B(p, r)|\right|}{|B(pe^{hX}, r)|\,|B(p, r)|}\|u\|_{L^1(B(p, r))} \nonumber \\
  &\leq \frac{|B(pe^{hX}, r) \vartriangle B(p, r)|}{|B(pe^{hX}, r)|}\|u\|_{L^\infty(\Om')}
  + \frac{\left||B(pe^{hX}, r)|-|B(p, r)|\right|}{|B(pe^{hX}, r)|\,|B(p, r)|}\|u\|_{L^1(\Om')}. \nonumber
 \end{align}
 In the last step we use the local boundedness, see Proposition~\ref{prop-dbl}(b), and the local integrability of $u$, see Definition~\ref{defn-harm}. Hence,
 \begin{align*}
  &\int_{\Om'} \left|\frac{u(pe^{hX})-u(p)}{h}\right|^s\,dq \\
  &\leq 2^s|\Om'|\left(\|u\|^s_{L^\infty(\Om')}+\|u\|^s_{L^1(\Om')}\right)
  \left(\left|\frac{|B(pe^{hX}, r) \vartriangle B(p,r)|}{h|B(pe^{hX}, r)|}\right|^s+
  \left|\frac{\left||B(pe^{hX}, r)|-|B(p, r)|\right|}{h|B(pe^{hX}, r)|\, |B(p, r)|}\right|^s
  \right).
 \end{align*}

  Note that the left-invariance of the Lebesgue measure $dq$ implies that the second term above vanishes.
  Moreover, since $d(p, pe^{hX})=|h|\mathcal{N}(X)$, then for sufficiently small $|h|$ we have:
  \[
   \left|B(pe^{hX}, r) \vartriangle B(p,r)\right|\leq \left|B(pe^{hX}, r+|h|\mathcal{N}(X))\setminus B(p,r-|h|\mathcal{N}(X))\right|.
  \]
  In a consequence, we obtain
  \begin{align*}
   \left|\frac{|B(pe^{hX}, r) \vartriangle B(p,r)|}{h|B(pe^{hX}, r)|}\right|  &=\left|\frac{(r+|h|\mathcal{N}(X))^Q-(r-|h|\mathcal{N}(X))^Q}{hr^Q}\right|
  \leq 2Q\frac{(r+|h|)^{Q-1}}{r^Q}\leq 2Q\frac{3^{Q-1}}{r},
  \end{align*}
 where in the second estimate we use the mean value theorem applied to function $t^Q$, for $t\in R_+$ on the interval $[r-|h|\mathcal{N}(X),r+|h|\mathcal{N}(X)]$. We choose radii of balls above such that $r=\frac12\dist(\Om',\bd \Om)$. Therefore,
 for sufficiently small $|h|$, we have the following estimate:
 \begin{align*}
 \int_{\Om'} \left|\frac{u(pe^{hX})-u(p)}{h}\right|^s\,dq
 &\leq 4^sQ^s|\Om'|\left(\|u\|^s_{L^\infty(\Om')}+\|u\|^s_{L^1(\Om')}\right) \left(\frac{(r+|h|)^{Q-1}}{r^Q}\right)^s \nonumber \\
 &\leq \left(2Q3^{Q-1}\right)^s\frac{(1+|\Om'|^s)|\Om'|\,\|u\|^s_{L^\infty(\Om')}}{\dist^s(\Om', \bd \Om)}.
  \end{align*}
 Lemma~\ref{lem-diff-quot} implies that $D_h^{X}u\to \tilde X(u)$ in $L^s(\Om')$.

 In particular, if $G=\Hein$, then the same convergence holds for $X_i$ with $i=1,2,\ldots, 2n$. The proof is completed. 

If $u$ is weakly harmonic in $\Om$, then all the above estimates hold provided that at the beginning of the proof we set $r=\frac12\dist(\Om',\bd \Om)$.
 \end{proof}

\subsection{Strongly harmonic functions on Carnot-Carath\'eodory groups are smooth}\label{subs-smooth}

The purpose of the following discussion is to show that the geometric definition of strongly harmonic functions, cf. Definition~\ref{defn-harm}, implies in the setting of Carnot--Carath\'eodory groups that such functions are $C^{\infty}$ smooth. In fact, we can prove this assertion for a wider class of functions, defined with respect to pseudodistances (i.e. quasimetrics) instead of metrics. Namely, suppose that $G$ is equipped with a pseudonorm $\mathcal{N}$ defining a pseudodistance $d$, cf. Definitions~\ref{defn-pseudon} and~\ref{defn-pseudod}. Furthermore, let $\Om\subset G$ be a domain and let $B(p, r)\Subset \Om$ denote any ball defined with respect to $d$. We extend Definition~\ref{defn-harm} and call a function $u:\Om\to R$ \emph{strongly harmonic} if, under the above assumptions, $u$ satisfies Definition~\ref{defn-harm}:

\begin{equation}\label{defn-subs-smooth}
  u(p)=\frac{1}{|B(p, r)|}\int_{B(p,r)} f(q) dq,\quad \hbox{ for all } p\in G.
\end{equation}

It turns out that strong harmonicity with respect to pseudoballs implies smoothness, see Theorem~\ref{thm-mvp-smooth}. This illustrates that our notion of strong harmonicity is robust and allows some degree of flexibility for the distance. Furthermore, we show that for the distance $d$ defining the fundamental solution of the $\opl$-harmonic operator on $G$, strongly harmonic functions are a subset of $\opl$-harmonic functions, see Theorem~\ref{thm-mvp}.

Before proving Theorem~\ref{thm-mvp-smooth} we need to recall the following auxiliary lemma on the integration in the polar coordinates in $G$ and bump functions.

\begin{lem}[cf. Proposition 1.25 in Folland--Stein~\cite{FolStn}]\label{polcrds}
Let $G$ be a Carnot group of the Hausdorff dimension $Q$. For any homogeneous norm $\mathcal{N}$ on $G$, there exists a unique Radon measure $d\sigma$ on the unit sphere $S(0,1) = \{ q\in G : \mathcal{N}(q) = 1\}$  giving the following polar coordinate expression of the integral:
\begin{align} \int_{G} u(q) \, d q  & = \int_0^\infty \int_{S(0,1)}  u( \delta_r(v) ) \, r^{Q-1}  \, d \sigma(v)\, dr. \label{polarcoordCarn}
\end{align}
\end{lem}

Next, for smoothing purposes we need some bump functions. If $\mathcal{N}$ is any fixed norm on $G$ and $k \in \mathbb{N}$, then the function $\psi: G\to \R$ defined by
\begin{align}
\psi(q)=
\begin{cases}
Ce^{-\frac{1}{1-\mathcal{N}(q)^k}}, & \mathcal{N}(q)<1\\
0, & \mathcal{N}(q) \geq 1,
\end{cases} \label{molif1}
\end{align}

where $C^{-1}= \int_{G} \psi(q)dq$ , has the following properties:
\begin{itemize}
	\item [(i)] $\psi$ is continuous and compactly supported in $B(0,1)$,
	\item [(ii)] $\psi$ is $C^\infty$ on  $G\setminus \{0\}$,
	\item [(iii)] $\psi$ is constant on all spheres centered at $0$,
	\item [(iv)] $\psi$ is $C_0^\infty$, if $k$ can be chosen so that $\mathcal{N}^k$ is $C^\infty$ at $0$.
\end{itemize}
Note that item (iv) is fulfilled by the Folland-Kaplan pseudonorm with $k=4$ and the pseudonorms defined at \eqref{canonN} with $k=2s!$. 

Alternatively, we can also use the bump function given by:
\begin{equation}\label{eq-psi}
\psi=\frac{1}{|B(0,1)|} \chi_{B(0,1)}.
\end{equation}
In this case we observe that
\begin{equation}\label{eq-conv}
u * \psi_\eps= \vint_{B(p,\eps)} u(q) dq.
\end{equation}
 Note that both $\psi$ defined in \eqref{molif1} and \eqref{eq-psi}  are constant on spheres, i.e., $\psi \circ \delta_r |_{S(0,1)}$ is constant.

\begin{theorem}\label{thm-mvp-smooth}
Let $\mathcal{N}$ be any pseudo-norm such that \eqref{molif1} is a $C_0^\infty$ bump for an appropriate choice of $k \in \mathbb{N}$. If $u\in L^1_{\loc}(G)$, then
	\begin{align}
	u(p)&=\vint_{B(p,R)} u(q) dq,\quad \hbox{ for all } p\in G \label{mvp2}
	\end{align}
	is equivalent to
	\begin{align}
	u(p) \int_{S(0,1)} d \sigma (v) &= \int_{S(0,1)}  u(p \delta_{r}( v) )  d \sigma(v),\quad \hbox{ for all } p\in G,  \label{mvp1}
	\end{align}
	where $B(p,R)$ and $S(p,R)$ are the balls and spheres, respectively, determined by $\mathcal{N}$. Furthermore, either of \eqref{mvp2} or \eqref{mvp1} implies that $u\in C^\infty(G)$.
\end{theorem}

\begin{proof} The implication \eqref{mvp1}$\implies$\eqref{mvp2} follows immediately from  \eqref{polarcoordCarn}. To check that also the opposite implication holds true, let us first observe that functions in $\harm(G)$ are continuous in $G$. Indeed, Proposition 4.1 in \cite{agg} stays that continuity of measure $\mu$ with respect to the underlying metric $d$ implies that functions in $\harm(\Om, \mu, d)$ are continuous, see also \eqref{m-cont-m}. The proof of this observation relies on the estimate similar to the one in the beginning of Theorem~\ref{thm-diff-harm-fun}. If $d$ is a pseudodistance determined by $\mathcal{N}$ instead of a distance, then we proceed as follows. Since group $G$ equipped with the Lebesgue measure $dq$ and the sub-Riemannian distance $d_s$ is geodesic and doubling, then by Proposition~\ref{prop-dbl}(a) we get, that $dq$ is continuous with respect to $d_s$. Remark~\ref{rem-psn} enables us to conclude that $dq$ is continuous also with respect pseudodistance $d$ and hence, Proposition 4.1 in \cite{agg} applies to $d$ as well. Thus, we conclude continuity of $u$ in $G$.

Next, we observe that if \eqref{mvp2} holds, then for $p\in G$ we have
$$
\int_{0}^{R} \left ( u(p)\int_{S(0,1)} d \sigma(v)-\int_{S(0,1)} u(p \delta_r(v))\, d \sigma(v) \right )r^{Q-1} dr =0 $$
and so it follows that
\begin{align}
\int_{R_1}^{R_2} \left ( u(p)\int_{S(0,1)} d \sigma(v)-\int_{S(0,1)} u(p \delta_r(v))\, d \sigma(v) \right )r^{Q-1} dr =0 \label{avinpol}
\end{align} for all admissible  $R_1$ and $R_2$ satisfying $0<R_1<R_2$. If $h(r)$ denotes the expression within the brackets in \eqref{avinpol}, then by the mean value theorem, we have $h(r)r^{Q-1}=0$ for all admissible  $r$, and so  $h(r)=0$ for all such $r$.

To check the converse, we first note that since $dq = dq^{-1}$, the mean value theorem also implies that
\begin{align*}
\int_{S(0,1)}  u( \delta_r(v) ) \,  d \sigma(v)=\int_{S(0,1)}  u( \delta_r(v^{-1}) ) \, d \sigma(v).
\end{align*}

We are now in a position to  show that \eqref{mvp1} results in $u\in C^\infty$. Let $\psi$ be the $C_0^\infty$ bump given by \eqref{molif1} and let $u \in L^1_{\loc}(G)$. Then
\begin{align}
	u * \psi_\eps(p) &= \int_{G} u(pq^{-1}) \psi_\eps(q)dq  \nonumber \\
	&= \int_{B(0,\eps)} u(pq^{-1}) \psi(\delta_{1/\eps}(q)) \frac{1}{\eps^Q} dq \nonumber  \\
	&= \int_{B(0,1)} u(p \delta_{\eps}(w^{-1})) \psi(w)dw \qquad (w:=\delta_{1/\eps}(q)) \nonumber \\
	&= \int_0^1 \int_{S(0,1)}  u(p \delta_{\eps r}( v^{-1})) \psi(\delta_r(v) ) r^{Q-1} d \sigma(v) dr \qquad (v:=\delta_{1/r}(w)) \nonumber \\
	&= \int_0^1 \int_{S(0,1)}  u(p \delta_{\eps r}( v^{-1}))  d \sigma(v) \, \eta(r) r^{Q-1} dr  \quad  \qquad (\eta(r):=\psi(\delta_r(v) )) \nonumber \\
	&= \int_0^1 \int_{S(0,1)}  u(p \delta_{\eps r}( v))  d \sigma(v) \, \eta(r) r^{Q-1} dr \nonumber \\
	&= u(p) \int_0^1 \int_{S(0,1)}  \eta(r) r^{Q-1} d \sigma(v) dr. \qquad  (\hbox{by }\eqref{mvp1}). \label{conveq1}
\end{align}

	Furthermore, we also have that
	\begin{align}
	1= \int_{B(0,1)} \psi(q) dq & =\int_0^1 \int_{S(0,1)} \psi( \delta_r(v)) r^{Q-1} d \sigma(v) dr \nonumber \\
	&= \int_0^1 \int_{S(0,1)} \eta(r) r^{Q-1} d \sigma(v) dr.   \label{conveq2}
	\end{align}

	Hence, it follows from \eqref{conveq1} and \eqref{conveq2} that for any $p\in G$
	\begin{align*}
	u * \psi_\eps(p)  &=u(p).
	\end{align*}

 This, by the smoothing properties of convolutions, implies that $u$ is $C^\infty$ and completes the proof of the theorem.
\end{proof}	

\begin{rem}\label{rem-weak-smooth}
 Let us observe that the technique used in the above proof does not give the $C^{\infty}$-regularity for weakly harmonic functions. Indeed, the use of convolutions $u * \psi_\eps$ requires the mean value property to hold for all radii in a given ball, a property that fails in general for functions satisfying Definition~\ref{defn-w-harm}.
\end{rem}

\subsection{Strongly harmonic functions are $\opl$-harmonic}\label{sub-Lharm}

Recall, that in this section we allow $d$ in Definition~\ref{defn-harm} to be a pseudodistance. In the proof we appeal to the $C^{\infty}$-regularity of strongly harmonic functions, cf. Theorem~\ref{thm-mvp-smooth}. However in fact $C^2$-regularity is enough for the result to hold. Recall the definition of an $\opl$-gauge as a pseudometric derived from the fundamental solution of the operator $\opl$ (see Formula~\ref{def-lgauge}).
\begin{thm} \label{thm-mvp}
 Let $\Om$ be a domain in a Carnot-Carath\'eodory group $G$ with the Hausdorff dimension $Q$. Let further $u:\Omega \to \mathbb{R}$ be strongly harmonic with respect to the balls given by an $\opl$-gauge $\N$. Then $u$ is $\mathcal{L}$-harmonic.
\end{thm}

\begin{proof}
   By the definition, $u\in L_{loc}^1(\Omega)$ and satisfies the mean-value property at every point $p\in \Om$ and any ball $B(p,r) \Subset \Om$. Let $d$ stand for a pseudometric defined by an $\opl$-gauge $\N$ as in Definition~\ref{defn-pseudod}. For any $s>0$ let us denote by $\Omega_s:=\{ p \in \Omega \, : \, d(p,\bd \Om )>s \}$. Set
   $$
    g_t := \frac{1}{|B(0,t)|} \chi_{B(0,t)}.
   $$
   Then, formulas \eqref{defn-subs-smooth} and \eqref{eq-conv} imply that for all admissible radii $0<r<s$ (i.e. such that $B(p, s)\Subset \Om$) we have the following equation:
\begin{equation}
 u*g_r(p)=u*g_s(p) \label{eq-eq}
\end{equation}
for all $p \in \Omega_s$.

\emph{Claim:} Let $\mathcal{L}^R$ be the right-invariant Laplacian corresponding to $\mathcal{L}$, cf. \eqref{sublap}. Then for any $0 < r < s$, the non-homogeneous equation
\begin{equation}
 \mathcal{L}^R w_{r,s}=g_{s}-g_{r}\label{eq-nonh}
\end{equation}
admits a $C^{1}$ solution with compact support in $B(0, s)$. In order to prove the claim we  adopt the proof of Lemma 4.1 in Gilbarg--Trudinger~\cite{gt} and refer to Theorem 4.5 in Ricciotti~\cite{Ricciotti} on the growth estimates for the fundamental solution and its horizontal gradient. Since both the result and the employed techniques are classical we restrict our discussion to sketch only.

Let $\Gamma^R$ be a fundamental solution of $\opl^R$ as above (cf. Theorem~\ref{thm-fund-exist}). We define
\begin{align}
w_{r,s}(p)&:=\int_{\Om} \Gamma^R(pq^{-1})(g_{s}-g_{r})(q)\,dq  \nonumber \\
&=\frac{1}{|B(0,s)|}\int_{B(0,s)\setminus B(0,r)} \Gamma^R(pq^{-1})\,dq+\left(\frac{1}{|B(0,s)|}-\frac{1}{|B(0,r)|}\right)\int_{B(0,r)} \Gamma^R(pq^{-1})\,dq. \label{eq-wrs}
\end{align}

Since $\Gamma^R\in L^1_{loc}(G)$ we have $w_{r,s}\in L^{1}_{loc}(\Om)$. Notice also, that $B(0,s)={\rm supp}\, w_{r,s}$. Moreover, since for any $s$ such that $0<r<s$, the set $\overline{\Om_s}$ is compact, \cite[Theorem 4.5]{Ricciotti} applies (cf. also Remark 6.18 in Capogna~\cite{cap2}). In particular, upon setting $\Gamma^R_p(q^{-1}):=\Gamma^R(pq^{-1})$ for any fixed $p\in B(0,s)$ inequality \cite[(4.71)]{Ricciotti} gives us the following estimate for the horizontal gradient $\nabla_0\Gamma^R_p(q^{-1})$ ($\nabla_{\mathbb{H},q^{-1}}\Gamma^R$ in the notation of \cite{Ricciotti}):
\[
|\nabla_{0} \Gamma^R_p(q^{-1})|\leq C\frac{d(p,q^{-1})}{|B_{d(p,q^{-1})}(p)|}.
\]
This estimate, the definition of $w_{r,s}$ in \eqref{eq-wrs} and standard telescopic argument imply that $\nabla_0 w_{r,s}\in L^1_{loc}(\Om)$. Indeed, let us fix any domain $\Om'\Subset G$ and observe that if $p\in B(0,s)$, then
$B(0,s)\subset B(p, 2s)$. Furthermore, recall that $dq=dq^{-1}$ and that domains $B(0,s)$ and $B(0,s)\setminus B(0,r)$ are symmetric with respect to the origin, hence both $q$ and $q^{-1}$ are contained in these domains. Therefore, for all $p\in \Om'$ we obtain the following estimate
\begin{align*}
 &|\nabla_{0} w_{r,s}(p)|\\
 &\leq \frac{1}{|B(0,s)|}\int_{B(0,s)\setminus B(0,r)} |\nabla_{0} \Gamma^R_p(q^{-1})|\,dq+\left(\frac{1}{|B(0,s)|}-\frac{1}{|B(0,r)|}\right)\int_{B(0,r)} |\nabla_{0}\Gamma^R_p(q^{-1})|\,dq\\
 & = \frac{1}{|B(0,s)|}\int_{B(0,s)\setminus B(0,r)} |\nabla_{0} \Gamma^R_p(q)|\,dq+\left(\frac{1}{|B(0,s)|}-\frac{1}{|B(0,r)|}\right)\int_{B(0,r)} |\nabla_{0}\Gamma^R_p(q)|\,dq\\
 &\leq \frac{1}{|B(0,s)|} \int_{B(0,s)\setminus B(0,r)} \frac{d(p,q)}{|B_{d(p,q)}(p)|}\,dq
 +\left(\frac{1}{|B(0,s)|}-\frac{1}{|B(0,r)|}\right)\int_{B(0,r)} \frac{d(p,q)}{|B_{d(p,q)}(p)|}\,dq \\
 &\leq \frac{1}{|B(0,1)|}\left(\,\frac{2}{s^Q}+\frac{1}{r^Q}\,\right)\int_{B(0,s)}\frac{d(p,q)}{|B_{d(p,q)}(p)|} \,dq \\
 &\leq \frac{1}{|B(0,1)|}\left(\,\frac{2}{s^Q}+\frac{1}{r^Q}\,\right)\int_{B(p,2s)}\frac{d(p,q)}{|B_{d(p,q)}(p)|} \,dq \\
 &\leq \frac{1}{|B(0,1)|}\left(\,\frac{2}{s^Q}+\frac{1}{r^Q}\,\right)\sum_{k=-1}^{\infty}\frac{s}{2^k}\int_{B(p,\frac{s}{2^k})\setminus B(p,\frac{s}{2^{k+1}})}\frac{1}{|B_{d(p,q)}(p)|} \,dq \\
 &\leq \frac{1}{|B(0,1)|}\left(\,\frac{2}{s^Q}+\frac{1}{r^Q}\,\right)\sum_{k=-1}^{\infty}\frac{s}{2^k} \frac{|B(p,\frac{s}{2^k})|}{|B(p, \frac{s}{2^{k+1}})|}\\
 &\leq \frac{8s}{|B(0,1)|}\left(\,\frac{2}{s^Q}+\frac{1}{r^Q}\,\right)
\end{align*}
where in order to obtain the last estimate we employ the doubling property of the measure. From this we infer that
\[
\int_{B(0,s)}|\nabla_{0} w_{r,s}(p)|\,dp\leq 4\left(s+\left(\frac{s}{r}\right)^Q\right).
\]
Hence, it follows that $|\nabla_{0} w_{r,s}|\in L^1_{loc}(G)$ since the domain $\Om'$ is arbitrary. In particular, if $r=s/2$ the $L^{1}_{loc}$-estimate is uniform in $s$ for $s\to 0^{+}$. We appeal to this observation later in this proof.

In fact, as we now show, $w_{r,s}\in C^{1}(\Om)$. This discussion allows us to compute
$\opl^R(w_{r,s})$ in a weak sense and obtain that $w_{r,s}$ is a solution to \eqref{eq-nonh}.

By employing the reasoning similar to \cite[Lemma 4.1]{gt}, for any $\ep>0$ we define
\[
 w_\ep(p)=\int_{\Om} \Gamma^R(p^{-1}q)(g_{s}-g_{r})\eta\left(\frac{\mathcal{N}(p^{-1}q)}{\ep}\right)\,dq,
\]
where $\eta\in C^1(\R)$, $0\leq \eta \leq 1$, $0\leq \eta'\leq 2$ and $\eta\equiv 0$ for all $t\leq 1$ whereas
for all $t\geq 2$ we require that $\eta\equiv 1$. Similarly to the Euclidean case one proves that $w_\ep\in C^1(G)$ and, furthermore, that
\[
 w_{\ep} \to w_{r,s}\quad \hbox{ and }\quad \nabla_0 w_{\ep} \to \int_{\Om} \left( \nabla_0 \Gamma^R(p^{-1}q)\right)\,(g_{s}-g_{r})\,dq,
\]
as $\ep \to 0$ uniformly on compacta in $\Om$. This completes the proof of the claim.

It now follows from Property 4 in Section~\ref{sect-conv} and \eqref{eq-eq} that
$$
(\mathcal{L}u )* w_{r,s}=u * \mathcal{L}^R  w_{r,s}=u * g_s-u * g_r =0
$$
for all $0 < r < s$.
Next we show that $\mathcal{L}u = 0$ on $\Omega$. In order to complete this goal, denote $\phi_s:=-w_{s/2,s}$. Then
$$
\mathcal{L}^R (s^{Q-2} \phi_s \circ \delta_s)  = \frac{1}{|B(0,1)|} \left ( (2^Q-1) \chi_{B(0,1/2)} -\chi_{B(0,1) \setminus B(0,1/2)} \right )=\mathcal{L}^R (\phi_1).
$$
A similar argument to that in Proposition 5.3.12 of \cite{blu} shows that
\[
 \Gamma^R(\delta_{s}(pq^{-1}))=s^{2-Q}\Gamma^R(pq^{-1}),
\]
and so the following identity holds:
\begin{align*}
&\int_G \phi_s(p)\,dp \\
&=\frac{1}{s^{Q-2}} \int_G  s^{Q-2} \phi_s \circ \delta_s(p) s^Q dp \\
& = s^{2} \int_{B(0,s)}\!\!\! s^{Q-2} \left[\frac{1}{|B(0,s)|}\int_{_{B(0,s)\setminus B(0,s/2)|}} \!\!\!\Gamma^R(\delta_s(p)q^{-1})\,dq+\left(\frac{1}{|B(0,s)|}-\frac{1}{|B(0,s/2)|}\right)\int_{_{B(0,s/2)}} \!\!\!\Gamma^R(\delta_s(p)q^{-1})\,dq\right]\,dp \\
& = s^{2} \int_{B(0,s)}\!\!\! s^{Q-2} \left[\frac{1}{|B(0,1)|}\int_{_{B(0,1)\setminus B(0,1/2)}} \!\!\!\Gamma^R(\delta_s(pq^{-1}))\,dq+\left(\frac{1}{|B(0,1)|}-\frac{1}{|B(0,1/2)|}\right)\int_{_{B(0,1/2)}} \!\!\!\Gamma^R(\delta_s(pq^{-1}))\,dq\right]\,dp \\
& = s^{2} \int_{B(0,s)} \left[\frac{1}{|B(0,1)|}\int_{_{B(0,1)\setminus B(0,1/2)}} \!\!\Gamma^R(pq^{-1})\,dq+\left(\frac{1}{|B(0,1)|}-\frac{1}{|B(0,1/2)|}\right)\int_{_{B(0,1/2)}} \!\!\Gamma^R(pq^{-1})\,dq\right]\,dp \\
& = s^2 \int_G \phi_1(p)\,dp.
\end{align*}
Let us denote $c:=\int_G \phi_1(q)dq$ and let $\Omega' \subset \Omega$ be compact. We define the following function:
$$
F^q(p)=
\begin{cases}
\mathcal{L}u(pq^{-1}), \quad & \hbox{for }p\in \Omega' q\\
0, \quad &{\rm otherwise}
\end{cases}
$$
and
$$
F(p)=\begin{cases}
\mathcal{L}u(p), \quad & \hbox{for }p\in \Omega'\\
0, \quad &{\rm otherwise}.
\end{cases}
$$
With this notation, we have
\begin{align*}
F* \frac{1}{s^2}\phi_s(p)- F(p) \int_G \phi_1(q)dq &=  \frac{1}{s^2} \int_G  (F(pq^{-1})- F(p))\phi_s(q) dq \\
& =  \frac{1}{s^2} \int_G (F^q(p)-F(p)) \phi_s(q) dq\\
& =  \frac{1}{s^2} \int_G (F^{\delta_s(q)}(p)-F(p)) \phi_s \circ \delta_s(q) s^Q dq\\
& =   \int_G (F^{\delta_s(q)}(p)-F(p)) \phi_1 (q) dq.
\end{align*}

By the Minkowski inequality applied to a fixed value of $1\leq \alpha <\infty$ we get the following estimate:

\begin{align*}
\| F* \frac{1}{s^2}\phi_s- cF\|_{L^\alpha(G)}
& \leq    \int_G \| F^{\delta_s(q)}-F \|_{L^\alpha(G)}  |\phi_1 (q)| dq.
\end{align*}

Recall that by Theorem~\ref{thm-mvp-smooth} we have $u\in C^{\infty}(G)$. This together with the definition of $F$ allows us to infer that $F\in L^\alpha(G)$ for all $1 \leq \alpha <\infty$. Since for all $p\in G$ it holds that $\| F^{\delta_s(q)}(p)-F(p)\|_{L^\alpha(G)} \leq 2\|F(p))\|_{L^\alpha(G)}$, the dominated convergence theorem implies that
$$
 \lim_{s \to 0}\| F* \frac{1}{s^2}\phi_s- cF\|_{L^\alpha(G)} = 0.
$$
Since $F* \frac{1}{s^2}\phi_s(p)=0$ we conclude that $\mathcal{L} u=0$ a.e. in  $\Omega'$. Therefore, we conclude that $\mathcal{L} u=0$ on  $\Omega$ as $\Omega'$ is an arbitrary compact subset of $\Omega$.
\end{proof}

 Theorem \ref{thm-mvp} shows that in Carnot groups, the strongly harmonic functions are a subfamily of the $\mathcal{L}$-harmonic functions. The opposite relation does not hold in general as demonstrated in Example~\ref{ex-sph} below, where a spherical harmonic polynomial, by definition satisfying the $\opl$-harmonic equation, is shown not to be strongly harmonic. We postpone this example till Section~\ref{sec-harm-hei1} and discuss the spherical harmonics in more detail there. Moreover, in Section~\ref{sec-harm-hei1} we identify a subclass of spherical harmonic polynomials in $\Hei$ which are strongly harmonic.

 We close this section with a consequence of Theorem \ref{thm-mvp}, the so-called three spheres theorem. This part of the presentation is based on \cite{aw1}. There, we show several variants of three-spheres theorems for sub-elliptic equations in Carnot groups of Heisenberg-type ($H$-type groups).

 The classical Hadamard three-circles theorem in $\R^2$ asserts that given three concentric circles with radii $0<R_1<R<R_2$ and a subharmonic function $u$ in the plane, the maximum of $u$ over a circle with radius $R$ is a convex function of $\log R$, with coefficients depending on the ratios of $R_1, R$ and $R_2$. The three-circles theorem has been generalized in various settings, including subharmonic functions in $\R^n$ for $n>2$, higher-dimensional concentric surfaces (e.g. three-spheres theorems), more general linear and quasilinear elliptic equations, the heat equation (three-parabolas theorem) and coupled elliptic systems of equations, see \cite{aw1} for further details and references.

 \noindent Let $G$ be an $H$-type group and $\Om\subset G$. For a function $u:\Om \to \R$ we define
\begin{equation*}
 M(r)= \sup \{ u(X) \, : \, X \in \Om, \, \,  |X|=r\}.
\end{equation*}

 The following observation holds.

 \begin{cor}\label{thm:3spheres-harm-peq2}
  Let $G$ be an $H$-type group and $\Om\subset G$ be a domain containing the identity element of $G$. Assume that $u: \Om \to \R$ is strongly harmonic in $\Om$ with respect to the $\opl$-gauge. Moreover, let us consider three concentric gauge-norm-spheres with radii $r_1<r<r_2$ contained in $\Om$.  Then
 \begin{equation}\label{eq:3pheres-lapl}
M(r)\leq M(r_1)\frac{r^{2-Q}-r_2^{2-Q}}{r_1^{2-Q}-r_2^{2-Q}} + M(r_2)\frac{r_1^{2-Q}-r^{2-Q}}{r_1^{2-Q}-r_2^{2-Q}}.
  \end{equation}
  Equality holds if and only if $u(X)\equiv \phi(r)$, where $r=|X|$ and $\phi$ is a function on the right-hand side of \eqref{eq:3pheres-lapl}.
\end{cor}
 The proof of the corollary is the direct consequence of Theorem~\ref{thm-mvp} above and Theorem 4 in \cite{aw1}.

\subsection{The converse to mean value property}\label{subs-conv}

 In this section we study the opposite problem to investigations in the previous section. Namely, suppose that a function obeys a mean value property with respect to the underlying Lebesgue (Hausdorff) measure and for balls in a given gauge $d$ (not necessarily the $\opL$-gauge). Can we then provide necessary and sufficient conditions for such a function to be sub(super)harmonic with respect to $\opL$? We restrict our discussion to the setting of the first Heisenberg group $\Hei$ due to complexity of the corresponding computations in $\Hein$ for $n\geq 2$.

 Let us motivate our studies with the following theorem. Most important conclusion of this result, especially relevant from our point of view, is that $\mathcal{L}$-harmonic functions need not in general be strongly harmonic. We formulate Theorem~\ref{thm-Lharm} in a setting of CC-groups, even though in what follows we will need this result only for the case $G=\Hei$. This illustrates that a relation between $\opl$-harmonicity and harmonic functions as in Definitions~\ref{defn-harm} and~\ref{defn-w-harm} is involved for all CC-groups and requires further studies, see also discussion in Section~\ref{sec-harm-hei1}.

 \begin{thm}\label{thm-Lharm}
  Let $G$ be a Carnot--Carath\'eodory group and $u:\Om\to \R$ be an $\opl$-harmonic function in a domain $\Om\subset G$. Furthermore, let $\mathcal{N}$ be a pseudonorm defined by the fundamental solution of $\opl$, cf.~\eqref{eq-N-fund}. Then, the following volume mean value property holds for all $p\in G$ and all balls $B(p,r)\Subset \Om$:
 \begin{align}
 u(p) &= \frac{\int_{B(p,r)} u(q)|\nabla_0 \mathcal{N}( p^{-1} q)|^2 dq }{\int_{B(p,r)}|\nabla_0 \mathcal{N}( p^{-1} q)|^2  dq} =\frac{\int_{B(0,r)}   u(pq) |\nabla_0 \mathcal{N}( q)|^2 dq }{\int_{B(0,r)}    |\nabla_0 \mathcal{N}(q)|^2  dq}. \label{mvpvol}
 \end{align}
 \end{thm}

In the Euclidean setting we have $|\nabla_0 \mathcal{N}|^2\equiv 1$ and \eqref{mvpvol} reduces to the mean value over Euclidean balls. In fact it is known, see \cite[Chapter 5]{blu}, that if $|\nabla_0 \mathcal{N}|^2$ is a constant then $G$ is commutative and thus the geometry is Euclidean. It follows that equivalence between $\mathcal{L}$-harmonicity and the strong harmonicity in noncommutative Carnot groups does not occur, see the discussion following the proof of Theorem~\ref{thm-mvp} and Example~\ref{ex-sph} below.

The proof of the theorem is based on, nowadays, classical techniques employed in the studies of the Carnot-Carath\'eodory groups. Nevertheless, we present it for the sake of completeness of the presentation and in order to demonstrate some crucial differences between the Euclidean and CC-settings. We refer to Appendix for the proof of Theorem~\ref{thm-Lharm}.

Let $\Om\subset \Hei\setminus\{0\}$ and let $u\in \harm(\Om, \mu)$, where $\mu=dq$ is the $3$-Lebesgue measure (equivalently the $3$-Hausdorff measure, denoted $\Hau$). Moreover, we assume that $d$ is a metric on $\Hei$. Therefore, it holds that
\begin{equation}\label{eq-mvp-lapl1}
 u(p)=\frac{1}{|B_d(p,r)|}\int_{B_d(p,r)}u(q)\,dq,
\end{equation}
for all $p\in \Om$ and every $r>0$ such that $B_d(p,r)\Subset \Om$. By Theorem~\ref{thm-mvp-smooth} we know that $u\in C^{\infty}(\Om)$. On the other hand, by Theorem 5.6.1 in \cite{blu} we have that a $C^2$ function satisfies the following mean value property with respect to an $\opL$-gauge, denoted $\dL$:
\begin{align}
 u(p)=\frac{1}{r^4 \int \limits_{B_{\dL}(0,1)}|\nabla_0 \dL|^2 dq}&\Bigg[\,\int \limits_{B_{\dL}(p,r)} u(q) |\nabla_0 \dL|^2(p^{-1}q) dq \nonumber \\
 &-2\int \limits_{B_{\dL}(p,r)} \int_{0}^{r} \varrho^3\bigg(\int \limits_{B_{\dL}(p,r)}\left(\frac{1}{d^2(p^{-1}q)}-\frac{1}{\varrho^2}\right)\, \opL u(q) dq\bigg) d\varrho \Bigg]. \label{eq-mvp-lapl2}
\end{align}

Observe that the first expression on the right-hand side of \eqref{eq-mvp-lapl2} is the ratio in \eqref{mvpvol} for $\mathcal{N}\equiv d_{\opl}$. Furthermore, the direct computations, cf. \cite[Chapter 5]{blu}, give us that in the standard notation $p=(z,t)$ for coordinates of a point $p\in \Hei$ it holds:
\[
|\nabla_0 d_{\opl}(z,t)|^2=\frac{|z|^2}{\sqrt{|z|^4 + t^2}}.
\]
Next, let us assume that $u$ is super (sub) solution, i.e. $\opL u\geq (\leq) 0$ in $\Om$, respectively. Then, by combining \eqref{mvpvol}, \eqref{eq-mvp-lapl1} together with \eqref{eq-mvp-lapl2} we obtain the following condition to be satisfied by $u$:
\begin{equation*}
 u(p)=\frac{1}{|B_d(p,r)|}\int_{B_d(p,r)}u(q)\,dq \leq (\geq) \frac{4}{\pi r^4} \int \limits_{B_{\dL}(p,r)} \frac{|z(p^{-1}q)|^2}{\sqrt{|z(p^{-1}q)|^4+|t(p^{-1}q)|^2}} \, u(q) dq,
\end{equation*}

where $z(\cdot)$ and $t(\cdot)$ stand for, respectively, the $z$- and $t$-coordinates of a point in $\Hei$, see Section~\ref{sec-harm-hei1}. Moreover, note that
\[
 \frac{|B_{\dL}(p,r)|}{|B_d(p,r)|}=\frac{\pi/4}{|B_d(0,1)|}:=C_{d,\dL} <\infty,
\]
due to the left-invariance of the Hausdorff measure on $\Hei$. Hence, the necessary condition for $u$ to satisfy $\opL u\geq (\leq) 0$ in $\Om$ is
\begin{equation*}
C_{d,\dL}\|u\|_{L^1(B_d(p,r))}\leq (\geq) \|u |\nabla_\opL \dL|^2(p^{-1}\circ\cdot)\|_{L^1(B_{\dL}(p,r))} \qquad \hbox{for all }p\in \Om.
\end{equation*}

\section{Harmonicity on $\Hei$}\label{sec-harm-hei1}

The purpose of this section is to provide a large class of $\opl$-harmonic functions which in the same time are also strongly harmonic with respect to the $\opl$-gauge distance. Namely, a subset of the so-called \emph{spherical harmonic polynomials}, called for short, spherical harmonics. It is, perhaps, surprising that such a class exists, if one takes into account that a spherical harmonic function must satisfy two kinds of the mean value property: the one in Definition~\ref{defn-harm} and \eqref{mvpvol}. First, we recall the necessary definitions and set up the stage for main computations of the mean value property for a class of spherical harmonics. Then, we present an example of a spherical harmonic function (and hence a $\opl$-harmonic function) which fails to be strongly harmonic, see Example~\ref{ex-sph}. Finally, we address an open question about identifying all spherical harmonics which are strongly harmonic.

  The discussion below is, in fact, valid for all the Heisenberg groups $\Hein$, however for simplicity we restrict to the case $n=1$. Recall, that on $\Hei$ one introduces the coordinates $(z,t)$ where $z=x+iy \in \mathbb{C}$, $t \in \mathbb{R}$ and the multiplication is defined by
\begin{align*}
(z_1,t_1)(z_2,t_2)&=(z_1+z_2, t_1+t_2 + 2\, {\rm Im}\,(z_1 \bar z_2)) \nonumber\\
&=(x_1 + x_2 , y_1 + y_2, t_1+t_2 + 2(x_2y_1-x_1y_2)). 
\end{align*}
We observe that $(z,t)^{-1}=(-z,-t)$.

  Since the Dirichlet problem for $\mathcal{L}$ is solvable on $B(0,1)$, and $\mathcal{L}$ is analytically hypoelliptic (see \cite{Grein} and the references there in), there is a family of $\mathcal{L}$-harmonic polynomials which play the same role in $\Hei$ as do the spherical harmonics in $\mathbb{R}^n$. By analytic hypoellipticity, any harmonic function on $B(0,1)$ is real analytic and the "spherical harmonics" are naturally defined in terms of their homogeneous degree.

\begin{defn}
An $\mathcal{L}$-spherical harmonic of degree $\ell = 0, 1, 2 , \dots$,  is a polynomial in $z$, $\bar z$ and $t$, which is $\mathcal{L}$-harmonic and homogeneous of degree $\ell$ with respect to the Heisenberg dilation.
\end{defn}

 Using Kor\'anyi's formula in \cite{kor}, a basis for the $\mathcal{L}$-harmonic polynomials of homogeneous degree $2m+k+l$  can be enumerated in the form
 $$ P_{k,l}^m(z,t)=r_{k,l}^m(t+i|z|^2,t-i|z|^2)z^k \bar z^l, $$
 where the polynomial $z^k \bar z^l$ is $\mathcal{L}$-harmonic and
\begin{equation*}
 r_{k,l}^m(w,\bar w)= m! \sum_{j=0}^m   C(l,j)C(k,m-j) w^{m-j} \bar{w}^j
\end{equation*}
with
\begin{equation}\label{eq2-sph-har}
C(l,j) = \begin{cases}
1 \quad  {\rm if} \quad j=0 \\
\frac{1}{j!}\prod_{i=0}^{j-1} (\frac{1}{2}+l+i) \quad  {\rm if} \quad j>0.
\end{cases}
\end{equation}
Hence, the $\mathcal{L}$-harmonic functions $u$ on $B(0,1)$ have the form
$$
u(z,t)=\sum_{k,l,m} a_{k,l,m}P_{k,l}^m(z,t).
$$

On $\Hei$, the $\mathcal{L}$-harmonic polynomials of the form $z^k \bar z^l$ are precisely those for which $k=0$ or $l=0$, and so a basis for the $\mathcal{L}$-harmonic polynomials of homogeneous degree $2m+k$ is given by elements of the form
\begin{align*}
P_{0,k}^m(z,t)= r_{0,k}^m(t+i|z|^2,t-i|z|^2)\bar z^k \quad {\rm or} \quad P_{k,0}^m(z,t)= r_{k,0}^m(t+i|z|^2,t-i|z|^2) z^k.
\end{align*}

We are now in a position to state the key observation, shown by direct calculation.
\begin{observ}
 The following spherical harmonic polynomials on $\Hei$ are strongly harmonic for all $k\in \mathbb{N}$:
\begin{align*}
P^1_{0,k}(z,t)&=((1+k)t+ik|z|^2)z^k \quad {\rm and} \quad P^1_{k,0}(z,t)=((1+k)t-ik|z|^2)\bar z^k.
\end{align*}
\end{observ}

\begin{proof} The proof relies on using the cylindrical coordinates. Notice, that since $P^1_{k,0}(z,t)= \overline{P^1_{0,k}(z,t)}$, it suffices to check that $P^1_{k,0}(z,t)$ is strongly harmonic for all $k$.

If $p=(z_0,t_0)$ and $q=(z,t)$, then we have
\begin{align*}
P^1_{0,k}(pq) &=\left((1+k)(t+t_0 +2 {\rm Im}(z_0 \bar z))+ik|z+z_0|^2\right)(z+z_0)^k.
\end{align*}
In cylindrical coordinates $(z,t)=(re^{i \theta},t)$ and the above function takes the following form:
\begin{align}
 P^1_{0,k}(p (re^{i \theta},t)) &=\big ((1+k) \big ( \, t_0+t+2r( \, y_0  \cos(\theta)-x_0  \sin(\theta) \, ) \, \big )(re^{i \theta}+z_0)^k \nonumber \\
& \quad + i\, k \big ( \, r^2 + 2r( x_0  \cos(\theta)+ y_0  \sin(\theta)) + x_0^2+y_0^2 \, \big ) (re^{i \theta}+z_0)^k. \label{intP1}
\end{align}
Moreover, for any ball $B(p,R)\subset G$ we have:
\begin{align*}
\int_{B(p,R)} P^1_{0,k}(q)dq &=\int_{B(0,R)} P^1_{0,k}(pq)dq =\int_0^R \int_{-\sqrt{R^4-r^4}}^{\sqrt{R^4-r^4}}\int_{-\pi}^{\pi} P^1_{0,k}(p (re^{i \theta},t)) \, d\theta\, dt\, r dr.
\end{align*}
In view of \eqref{intP1}, inner integral $\int_{-\pi}^{\pi} P^1_{0,k}(p (re^{i \theta},t)) \, d\theta$ requires the following integrals which we evaluate with residues:
\begin{align*}
& \int_{-\pi}^{\pi} (re^{i\theta} +z_0)^k d \theta  
 = \int_{S(0,r)} (z+z_0)^k \, \frac{dz}{iz}
 = 2 \pi z_0^k, \\
& \int_{-\pi}^{\pi} \cos(\theta)(re^{i\theta} +z_0)^k d \theta = r^k \int_{-\pi}^{\pi} \cos(\theta) (e^{i\theta} +\frac{z_0}{r})^k d \theta = r^k\int_{S(0,1)} \frac{1}{2} (z+ z^{-1} ) (z+ r^{-1} z_0 )^k \,  \frac{dz}{iz}
=k\pi r z_0^{k-1},\\
&\int_{-\pi}^{\pi} \sin(\theta)(re^{i\theta} +z_0)^k d \theta = r^k\int_{S(0,1)} \frac{1}{2i} (z- z^{-1} ) (z+ r^{-1} z_0 )^k \,  \frac{dz}{iz} =i\, k\pi r z_0^{k-1}.
\end{align*}
It now follows from \eqref{intP1} that
\begin{align*}
\int_{-\pi}^{\pi} P^1_{0,k}(p (re^{i \theta},t)) \, d\theta&= 2 \pi \Big ( (1+k) (t_0+t)  \, + \,  i   \ \,k|z_0|^2 \, \Big )z_0^k.
\end{align*}
The remaining integrations are straight forward, as we have
\begin{align*}
\int_{-\sqrt{R^4-r^4}}^{\sqrt{R^4-r^4}} \int_{-\pi}^{\pi} P^1_{0,k}(p (re^{i \theta},t)) \, d\theta dt &=  4 \pi  \Big (  (k+1)t_0  + i\,   k|z_0|^2 \Big )z_0^k\sqrt{R^4-r^4}
\end{align*}
and
\begin{align*}
\int_0^R \int_{-\sqrt{R^4-r^4}}^{\sqrt{R^4-r^4}} \int_{-\pi}^{\pi} P^1_{0,k}(p (re^{i \theta},t)) \, d\theta dt r dr &=  4 \pi  \Big (  (k+1)t_0  + i\,   k|z_0|^2 \Big )z_0^k\int_0^R \sqrt{R^4-r^4}\,r dr\\
&=  4 \pi  \Big (  (k+1)t_0  + i\,   k|z_0|^2 \Big )z_0^k \frac{\pi}{8}R^4.
\end{align*}
Since $$ \int_{B(p,R)}dq = \frac{\pi^2}{2}R^4 $$ we obtain that
\[
 P^1_{0,k}((z_0,t_0))=\vint_{B(p,R)} P^1_{0,k}(q)dq
\]
and hence, $P^1_{0,k}$ is strongly harmonic in $\Hei$.
\end{proof}

\begin{ex}\label{ex-sph}

The first spherical harmonic that is not strongly harmonic is
$$
P^2_{0,0}(z,t)=2t^2-|z|^4.
$$
Indeed, in this case computations similar to the one in the proof of the observation reveal that
$$
 \vint_{B(p,R)}P^2_{0,0}(q)dq = P^2_{0,0}(p)+ \frac{R^4}{4}.
$$
\end{ex}

Computation with MAPLE up to homogenous degree $40$ revealed no strongly harmonic spherical harmonics with the order of $t$ greater than $1$. Thus, one might suspect that strongly harmonic spherical harmonics are precisely those with the order of $t$ is less or equal to $1$ and a computational proof similar to the above might reveal it to be true. However, the computations become cumbersome due to the combinatorics arising from the coefficients $C(l,j)$ at \eqref{eq2-sph-har} and repeated use of the binomial formula. Considering the mean values at $0$ does not simplify the task, indeed  all the spherical harmonics are strongly harmonic at $0$.

For example let us consider
\begin{align*}
P_{k,0}^m(z, t)= r_{k,0}^m \big ( t +ir^2 ,  t -ir^2 \big ) z^k.
\end{align*}

As above, let us apply the cylindrical coordinates $(re^{i \theta},t)$ and denote $z:=rw$ for $w \in S(0,1)$.
Upon considering the mean value of $P_{k,0}^m$ at $p=0$, it follows that
\begin{align*}
\int_{-\pi}^\pi P_{k,0}^m(re^{i\theta},t)d \theta &=   \int_{S(0,1)} P_{k,0}^m(rw,t)\, \frac{dw}{iw}
= -i r^{k} r_{k,0}^m \big ( t +ir^2 ,  t -ir^2 \big ) \int_{S(0,1)} w^{k-1} \, dw
=0, \quad \hbox{for }k \geq 1.
\end{align*}
 Therefore, the mean value property for $P_{k,0}^m$ holds at the origin.

 It is interesting to note that if the strongly harmonic spherical harmonics are exactly the spherical harmonics of $t$ degree less or equal to one, then strongly harmonic spherical harmonics are solutions to the the Laplace-Beltrami equation $$ \tilde X^2 u + \tilde Y^2 u + \frac{\partial^2 u}{ \partial t^2}=0.$$ Moreover, by analytic hypoellipticity, the same can be said for strongly harmonic functions on $B(0,1)$.

  In view of results of this section we pose the following problems.

\smallskip
\noindent {\bf Open problems}
\begin{itemize}
\item [(1)] Identify all spherical harmonic polynomials $P_{k,0}^m$ that are strongly harmonic.
\item[(2)] Describe other classes of functions in Carnot-Carat\'eodory groups that are both $\opl$-harmonic and strongly harmonic.
\end{itemize}

\section{Determining set}\label{sec-det}

 Let $\Om\subset \Hei$. We say that $S\subset \Om$ is the \emph{determining set} if for any $u:\Om\to \R$ strongly harmonic in set $S$, it follows  that $u\in \harm(\Om)$. For the studies of determining sets for harmonic function in the Euclidean setting we refer to Flatto~\cite{Flatto}. In this section we show that for continuous functions it is enough to assume the mean value property on a dense subset of $\Om$ in order to infer the harmonicity in the whole domain $\Om$. The additional technical  assumption is that $\Hei$ must be a geodesic space with respect to the underlying metric $d$.

 \begin{observ}
  Let $S\subset \Om$ be dense in $\Om$. Furthermore, let us suppose that metric $d$ on $\Hei$ is such that $\Hei$ is geodesic space with respect to $d$. Then it holds that, if $u\in \harm(S)\cap C(\Om)$, then $u\in \harm(\Om)$.
 \end{observ}

 \begin{proof}
  Let $p\in \Om\setminus S$ and $(p_n)$ be a sequence of points in $S$ such that $p_n\to p$, as $n\to \infty$ with respect to the given metric $d$. Let $\epsilon>0$ and $N$ be such that for all $n>N$ it holds that
  \[
   -\epsilon+u(p_n)\leq u(p) \leq \epsilon+u(p_n),
  \]
  by continuity of $u$ at $p$. Then, the following estimate holds for any $r>0$ and all balls $B(p_n,r)\Subset \Om$, $B(p, r)\Subset \Om$:
  \begin{align}
   u(p)\leq \ep+u(p_n)&=\ep+\frac{1}{|B_d(p,r)|}\frac{|B_d(p,r)|}{|B_d(p_n,r)|}\left(\int_{B_d(p,r)}u(q)+\int_{B_d(p,r)\setminus B_d(p_n,r)}u(q)\right) \nonumber \\
   &\leq \ep+\frac{|B_d(p,r)|}{|B_d(p_n,r)|} u_{B_d(p,r)}+\int_{B_d(p,r)\Delta B_d(p_n,r)}u(q), \label{observ-aux}
  \end{align}
  where $B_d(p,r)\Delta B_d(p_n,r)$ stands for the symmetric difference between balls. Observe that $\frac{|B_d(p,r)|}{|B_d(p_n,r)|}\to 1$, as $n\to \infty$. Moreover, since $\Hei$, equipped with metric $d$, is a geodesic space, then by Proposition~\ref{prop-dbl}, the Lebesgue measure is metrically continuous with respect to $d$, cf. \eqref{m-cont-m} above, see also Definition 2.1 and the discussion in \cite[Section 2]{agg} for further details. In particular, the following symmetric difference satisfies condition \eqref{m-cont-m}:
  \[
  |B_d(p,r)\,\Delta\, B_d(p_n,r)|\to 0,\quad \hbox{ as }\quad n\to\infty.
  \]
  By applying the Lebesgue convergence theorem we conclude that the second integral in \eqref{observ-aux} converges to $0$, while the first one to mean value of $u$ at $p$. Note that, since $u\in L^1_{loc}(\Om)$, then by invoking again the metric continuity of the measure with respect to $d$, one obtains that
  \[
  \left|\int_{B_d(p,r)\Delta B_d(p_n,r)}u(q)\right|\to 0,\qquad \hbox{ as } n\to \infty.
  \]
  Thus, we also get that $u(p)\geq -\ep+u_{B_d(p,r)}$ and the arbitrary choice of $\ep$ allows us to conclude that $u(p)=u_{B_d(p,r)}$. Hence, the proof is completed.
  \end{proof}

\appendix
\section{Appendix - the proof of Theorem~\ref{thm-Lharm}}

In the appendix we prove Theorem~\ref{thm-Lharm}, which states that $\mathcal{L}$-harmonic functions satisfy a variant of the mean value property with respect to kernels defined via gradient of pseudonorm given by the fundamental solution of $\opl$. We begin with preliminaries regarding the representation of $\opl$ in the divergence form.

 Fix an orthonormal basis $\{ E_i \}_{i=1}^N$ such that $\g^1={\rm span} \, \{ E_i \, | \,  i = 1,\dots, N_1 \}$
 and
 $$
 \tilde X_i  u(p) =\frac{d}{dt}u(p\exp(t E_i ))|_{t=0}\quad \hbox{ for }i=1,2,\ldots, N.
 $$
 Recall that $N=N_1+\ldots+N_s$, cf. Definition~\ref{defn-basis}. If $x_i$, for $i=1,2,\ldots, N$ denote the coordinates on $G$ induced by the chosen basis of $\mathfrak{g}$ via the exponential map, then
  $$
  \tilde X_i u=\sum_{j=1}^{N} dx_{j} (\tilde X_{i}) \dbd{}{x_{j}}.
  $$
 In coordinates we have
 \begin{align}
  \mathcal{L}u  = \sum_{i=i_0}^{N_1}  \tilde X_i^2u &= \sum_{j=1}^{N} \sum_{k=1}^{N}   \sum_{i=1}^{N_1}  dx_{j}(\tilde X_{i}) \dbd{}{x_{j}} \left ( dx_{k}(\tilde X_{i}) \dbd{u}{x_{k}}\right ) \nonumber \\
 &=  \sum_{j=1}^{N} \sum_{k=1}^{N}   \sum_{i=1}^{N_1}  \dbd{}{x_{j}}\left ( dx_{j}(\tilde X_{i})  dx_{k}(\tilde X_{i})\dbd{u}{x_{k}}  \right ) \label{eq-kern}\\
 &={\rm div} A\nabla u,  \nonumber
 \end{align}
where the coefficients of matrix $A$ are given by the formula $A_{j, k} =\sum_{i =1}^{N_1} dx_j (\tilde X_i)  dx_k (\tilde X_i)$ for $j, k=1,\ldots, N$. Note that the third equality \eqref{eq-kern} uses the identity $\dbd{}{x_{j}} \left (dx_{j}(\tilde X_{i}) \right )=0$ for each $i=1,\ldots, N_1$ which follows from the nilpotency of $G$.

Expanding in coordinates as above we also get $\langle \nabla_0 u , \nabla_0 u \rangle_{G} = \langle A \nabla u   \nabla u \rangle_{\R^N}$ and note that a mapping $(p,q) \to \langle \nabla_0 u , \nabla_0 u \rangle (p^{-1} q)$ is left invariant and homogeneous of degree $0$ with respect to dilation since $\nabla_0 u$ has degree $0$.

\begin{proof}[Proof of Theorem~\ref{thm-Lharm}]  By direct calculation, it is easy to check that the following Green's identity holds:
 \begin{align*} v  \mathcal{L}u -u  \mathcal{L}v= {\rm div} \left (  v A\nabla u - u A \nabla v \right ).
 \end{align*}

  Since $u$ is $\mathcal{L}$-harmonic in $\Om\subset G$ and $v=\mathcal{N}^{2-Q} \circ \tau_{p^{-1}} =-\Gamma \circ \tau_{p^{-1}}$, then $v$ is $\mathcal{L}$-harmonic on $G\setminus \{p\}$ and
 \begin{align*}
  \int_{B(p,r) \setminus B(p,\eps)} {\rm div} \left ( v A\nabla u - u A \nabla v    \right )(q) dq=0.
 \end{align*}
  Applying Stokes theorem gives
 \begin{align} \label{inv1}
 \int_{\partial B(p,r)}    u A \nabla v (\nu) \cdot n(p,\nu) d S(\nu) = \int_{\partial B(p,\eps)}   u A \nabla v (\nu) \cdot n(p,\nu) d S(\nu),
 \end{align}
 where $dS$ is defined via the Gramm determinant. Note that we have used the following consequences of the choice of $u$ and $v$:
 \begin{align*}
 \int_{\partial B(p,\rho)}    v A \nabla u (\nu) \cdot n(p,\nu) d S(\nu) &=\rho^{2-Q}\int_{\partial B(p,\rho)}    A \nabla u (\nu) \cdot n(p,\nu) d S(\nu)\\
 &=\rho^{2-Q}\int_{\partial B(p,\rho)}   {\rm div} A \nabla u (q)  dq\\
 &=\rho^{2-Q}\int_{\partial B(p,\rho)}   \mathcal{L}u (q)  dq =0.
 \end{align*}
We define the following kernel $K(p,q)$, cf. Definition 5.5.1 and the proof of Theorem 5.5.4 in \cite{blu}:
 \begin{align*}
 K(p,q): = (2-Q)\mathcal{N}(p^{-1}q)^{1-Q}  \frac{\langle \nabla_0 \mathcal{N}( p^{-1} q), \nabla_0 \mathcal{N}( p^{-1} q) \rangle}{\|\nabla (\mathcal{N} \circ\tau_{p^{-1}})(q)\| } := (2-Q)\mathcal{N}(p^{-1}q)^{1-Q}  \mathcal{K} (p,q),
 \end{align*}
 where $p,q\in G$ and $p\not=q$.
 
 Next, we define
 \begin{align*}
 T_r(u)(p):=\int_{\partial B(p,r)}   u(q) K(p,q) dS(q)\quad \hbox{ for any }p\in G,
 \end{align*}
 where $r>0$ is such that $B(p,r) \Subset \Omega$. It turns out that transform $T_r$ satisfies mean value property. Namely, by \eqref{inv1} we have $T_r(u)(p)=T_\eps(u)(p)$ for all $0<\eps <r $. It then follows that
 \begin{align*}
 |T_r(u)(p)-u(p)T_r(\chi_\Omega)(p)| & \leq  \int_{\partial B(p,\eps)} |u(q)-u(p)| |K(p,q)|  dS(q)
  \leq  \frac{1}{Q-2}\sup_{\partial B(p,\eps)} |u(q)-u(p)||T_r(\chi_\Omega)(p)|.
 \end{align*}
 The continuity of $u$ implies the following: 
 \begin{align}
 T_r(u)(p) = u(p)T_r(\chi_\Omega)(p).  \label{TrMean}
 \end{align}
 Furthermore, since $T_r(u)(p)=T_t(u)(p)$ for all $t <r $, we have
 \begin{align*}
 T_r(u)(p)r^Q&=T_r(u)(p)\int_0^rQt^{Q-1}dt = \int_0^r T_t(u)(p)Qt^{Q-1}dt\\
 &=\int_0^r \int_{\partial B(p,t)}   u(\nu) K(p,\nu) dS(\nu)Qt^{Q-1}dt\\
 &=Q(2-Q)\int_0^r \int_{\partial B(p,t)}   u(\nu) \mathcal{K}(p,\nu) dS(\nu) dt\\
 &=Q(2-Q)\int_{B(p,r)}   u(q) \mathcal{K}(p,q) \|\nabla (\mathcal{N} \circ \tau_{p^{-1}})(q)\| dq \\
 &=Q(2-Q)\int_{B(p,r)}   u(q) \langle \nabla_0 \mathcal{N}( p^{-1} q), \nabla_0 \mathcal{N}( p^{-1} q) \rangle dq.
 \end{align*}
 Finally, from \eqref{TrMean} we get the assertion of Theorem~\ref{thm-Lharm} for all $p\in G$ and all balls $B(p,r)\Subset \Om$:
 \begin{align*}
 u(p) &= \frac{\int_{B(p,r)}   u(q)|\nabla_0 \mathcal{N}( p^{-1} q)|^2 dq }{\int_{B(p,r)} |\nabla_0 \mathcal{N}( p^{-1} q)|^2  dq} =\frac{\int_{B(0,r)}   u(pq) |\nabla_0 \mathcal{N}( q)|^2 dq }{\int_{B(0,r)} |\nabla_0 \mathcal{N}(q)|^2  dq}.
 \end{align*}
\end{proof}

 As mentioned in Section~\ref{sect5}, these computations indicate that $\mathcal{L}$-harmonic functions need not necessarily be strongly harmonic.


\begin{thebibliography}{99}

\bibitem{agg} {\sc T. Adamowicz, M. Gaczkowski, P. G\'orka}, \emph{Harmonic functions on metric measure spaces}, submitted, arxix:1601.03919.

\bibitem{aw1} {\sc T. Adamowicz, B. Warhurst}, \emph{Three-spheres theorems for subelliptic quasilinear equations in Carnot groups of Heisenberg-type}, Proc. Amer. Math. Soc. 144(10) (2016), 4291-4302.

\bibitem{bb} {\sc A. Bj\"orn, J. Bj\"orn}, \emph{Nonlinear Potential Theory on Metric Spaces}, EMS Tracts in Mathematics, 17, European Math. Soc., Zurich, 2011.


\bibitem{blu} {\sc A. Bonfiglioli, E. Lanconelli, F. Uguzzoni}, \emph{Stratified Lie Groups and Potential Theory for Their Sub-Laplacians}, Springer Monographs in Mathematics, Springer, 2007.

\bibitem{cap} {\sc L. Capogna}, \emph{Regularity of quasi-linear equations in the Heisenberg group}, Comm. Pure Appl. Math. 50(9) (1997), 867-889.

\bibitem{cap2} {\sc L. Capogna}, \emph{Regularity for quasilinear equations and $1$-quasiconformal maps in Carnot groups}, Math. Ann. 313(2) (1999), 263-295.

\bibitem{capcow} {\sc L. Capogna, M. Cowling}, \emph{Conformality and Q-harmonicity in Carnot groups}, Duke Math. J. 135(3) (2006), 455--479.

\bibitem{Flatto} {\sc L. Flatto}, \emph{The Converse of Gauss's Theorem for Harmonic Functions}, J. Diff. Eq. 1 (1965), 483-490.

\bibitem{Fol} {\sc G.B. Folland}, \emph{Subelliptic estimates and function spaces on nilpotent Lie groups}, Ark. Mat. 13(2) (1975), 161-207.

\bibitem{FolStn} {\sc G. B. Folland, E. M. Stein}, \emph{Hardy Spaces on Homogeneous groups}, Mathematical Notes, Princeton University, (1982).

\bibitem{GG} {\sc M. Gaczkowski, P. G\'orka}, \emph{Harmonic Functions on Metric Measure Spaces: Convergence and Compactness}, Potential Anal. 31, 203--214 (2009).

\bibitem{gt} {\sc D.\ Gilbarg, N.\ S.\ Trudinger}, \emph{Elliptic Partial Differential Equations of Second Order}, Springer-Verlag, 1983.

\bibitem{Giusti} {\sc E.\ Giusti}, \emph{Direct Methods in the Calculus of Variations}, World Scientific, Singapore, 2003.

\bibitem{Grein} {\sc P. C. Greiner}, \emph{Spherical harmonics on the Heisenberg group},  Canad. Math. Bull. 23(4) (1980), 383-396.

\bibitem{hor} {\sc L. H\"ormander}, \emph{Hypoelliptic second order differential equations}, Acta Math. 119 (1967), 147-71.

\bibitem{kel} {\sc O. Kellogg}, \emph{Converses of Gauss' theorem on the arithmetic mean}, Trans. Amer. Math. Soc. 36(2) (1934), 227-242.


\bibitem{kor} {\sc A. Kor\'anyi}, \emph{Kelvin Transforms and Harmonic Polynomials on the Heisenberg Group}, J. Funct. Anal. 45(2) (1982), 293-296.

\bibitem{kr1} {\sc A. Kor\'anyi, H. M. Reimann}, \emph{Quasiconformal mappings on the Heisenberg group}, Invent. Math. 80(2) (1985),  309-338.

\bibitem{kr2} {\sc A. Kor\'anyi, H. M. Reimann}, \emph{Foundations for the theory of quasiconformal mappings on the Heisenberg group}, Adv. Math. 111(1) (1995), 1-87.

\bibitem{minm} {\sc J. Manfredi, G. Mingione}, \emph{Regularity results for quasilinear elliptic equations in the Heisenberg group}, Math. Ann. 339(3) (2007), 485-544.

\bibitem{Ricciotti} {\sc D. Ricciotti}, \emph {$p$-Laplace Equation in the Heisenberg Group. Regularity of solutions}, SpringerBriefs in Mathematics. BCAM SpringerBriefs. Springer, [Cham]; BCAM Basque Center for Applied Mathematics, Bilbao, (2015), xiv+87 pp.
\end{thebibliography}
\end{document}